\UseRawInputEncoding 
\documentclass[11pt,a4paper]{article}
\usepackage{amsfonts}
\usepackage{amssymb}
\usepackage{mathrsfs}
\usepackage{amsmath}
\usepackage{booktabs}
\usepackage{epsf,epsfig,amsfonts,amsgen,indentfirst}
\usepackage{amsmath,amstext,amsbsy,amsopn,amsthm,bbding,wasysym}
\usepackage{multicol,mathdots}
\usepackage{subfigure}
\usepackage[numbers,sort&compress]{natbib}
\allowdisplaybreaks

\newtheorem{theorem}{Theorem}[section]

\newtheorem{lemma}[theorem]{Lemma}

\newtheorem{problem}[theorem]{Problem}

\begin{document}
\title
{\LARGE \textbf{Solution to an open problem on the computational complexity of immanant
\thanks{Supported by NSFC (Nos. 12261071,  12571019) and NSF of Qinghai Province (No. 2025-ZJ-902T).} }}

\author{Xiangshuai Dong$^a$, Tingzeng Wu$^{b,c}$\thanks{{Corresponding author.\newline
\emph{E-mail address}: mathtzwu@163.com, a3566293588@163.com,  gaoxing@lzu.edu.cn
}}, Xing Gao$^{d,e}$ \\
{\small $^{a}$ School of Mathematical Sciences, Xiamen University,  }\\
{\small  Xiamen, 361005,  P.R.~China} \\
{\small $^{b}$ School of Mathematics and Statistics, Qinghai Minzu University, }\\
{\small  Xining, Qinghai 810007, P.R.~China} \\
{\small $^{c}$ Qinghai Institute of Applied Mathematics, Xining, Qinghai 810007, P.R.~China}\\
{\small $^{d}$ School of Mathematics and Statistics, Lanzhou University, }\\
{\small  Lanzhou, Gansu, 730000, P.R.~China}\\
{\small $^{e}$ Gansu Provincial Research Center for Basic Disciplines }\\
{\small   of Mathematics
and Statistics, Lanzhou, Gansu, 730070, P.R.~China}}
\date{}

\maketitle
\noindent {\bf Abstract:}
Immanants are a class of generalized matrix functions associated with the irreducible characters of the symmetric group. B\"{u}rgisser [SIAM J. Comput., 30 (2000), pp. 1023--1040] proved that the computation of hook immanants and immanants corresponding to rectangular Young diagrams of polynomially growing width is VNP-complete under $p$-projections. And  he posed an open problem: whether the family of immanants corresponding to rectangular Young diagrams of width $2$ is VNP-complete under $p$-projections. This paper gives a solution to this  problem. 
We prove that, over any field of characteristic zero, the immanant families associated with rectangular Young diagrams of width $2$ and of length $2$ are both VNP-complete under $p$-projections.


\section{Introduction}%
The definitions of computational complexity are given in the literature \cite{Bur1,Bur2,val}. Let $\mathbb{F}$ be a field of characteristic zero, and write $[n]=\{1,\ldots,n\}$. A {\em $p$-family} is a sequence of polynomials $(f_n)$, where $f_n\in \mathbb{F}[X_1, \ldots, X_{v(n)}]$, and the number of variables $v(n)$ and the degree $\deg f_n$ are both {\em $p$-bounded} functions of $n$, i.e., they are bounded by a polynomial in $n$. Let $L(f_n)$ denote the {\em complexity} of the polynomial $f_n \in \mathbb{F}[X_1, \ldots, X_{v(n)}]$, i.e., the minimum number of arithmetic operations (addition, subtraction, multiplication) needed to compute $f_n$ from the variables $X_i$ and constants in $\mathbb{F}$ using a straight‐line program.
A $p$-family $(f_n)$ is called {\em $p$-computable} if $L(f_n)$ is polynomially bounded in $n$. A $p$-family $(f_n)$ is called {\em $p$-definable} if there exists a $p$-bounded function $u(n)$ and a $p$-computable family $(g_n)$, with $g_n\in \mathbb{F}[X_1,\ldots,X_{v(n)},Y_1,\ldots,Y_{u(n)}]$, such that for all $n$,
$$
f_n(X_1, \ldots, X_{v(n)})
=\sum_{e=(e_1,\ldots,e_{u(n)})\in\{0,1\}^{u(n)}}
g_n(X_1,\ldots,X_{v(n)},e_1,\ldots,e_{u(n)}).
$$
All $p$-definable families form the complexity class $\mathrm{VNP}$. Let $(g_m)$ be another $p$-family, where $g_m\in \mathbb{F}[Y_1,\ldots,Y_{u(m)}]$. If there exists a $p$-bounded function $t:\mathbb{N}\to\mathbb{N}$ and substitutions $a_i\in \mathbb{F}\cup\{X_1,\ldots,X_{v(n)}\}$ such that
$$
f_n(X_1,\ldots,X_{v(n)})
=g_{t(n)}(a_1,\ldots,a_{u(t(n))}),
$$
then $(f_n)$ is called a {\em $p$-projection} of $(g_m)$, denoted $(f_n)\leq_p(g_m)$. In other words, each $f_n$ can be obtained from $g_{t(n)}$ by substituting variables with either original variables or constants from the field. If $(g_m)\in \mathrm{VNP}$ and every $(f_n)\in \mathrm{VNP}$ satisfies $(f_n)\leq_p(g_m)$, then $(g_m)$ is said to be {\em $\mathrm{VNP}$-complete} under $p$-projections. This paper uses the usual non‐uniform algebraic complexity model, so each member of a family may use field constants that depend on the family index. The definition of {\em $c$-reduction} is given in \cite{rug}. For a polynomial $f$, the complexity $L^g(f)$ with oracle $g$ is defined as the minimum number of arithmetic gates plus the number of queries to $g$ on previously computed values, needed to compute $f$ from variables and constants in $\mathbb{F}$. A $p$-family $(f_n)$ is reducible to $(g_m)$ under {\em $c$-reduction} if there exists a $p$-bounded function $t$ such that $L^{g_{t(n)}}(f_n)$ is polynomially bounded in $n$.

A {\em partition} of an integer $m$ is a sequence of positive integers $\lambda = (\lambda_1, \ldots, \lambda_s)$ with $\lambda_1\geq \cdots\geq\lambda_s$ and $\sum_i \lambda_i = m$, written $\lambda \vdash m$. A partition $\lambda = (\lambda_1, \ldots, \lambda_s)$ is in bijection with its {\em Young diagram} $\{(i,j)\mid 1\leq i\leq s,\ 1\leq j\leq\lambda_i\}$, which is a left‐justified arrangement of finitely many equal‐sized squares, where the $i$-th row has $\lambda_i$ squares. We call $|\lambda| := \sum_i \lambda_i$ the {\em size} of $\lambda$, and $s$ and $\lambda_1$ the {\em length} and {\em width} of $\lambda$, respectively. If a Young diagram corresponds to a partition with $\lambda_1 = \cdots = \lambda_s$, it is called a {\em rectangular} Young diagram. For a partition $\lambda$, denote by $\lambda^\top$ its {\em conjugate partition}; the Young diagram of $\lambda^\top$ is the transpose of that of $\lambda$. For example, the rectangular Young diagram of width $2$ corresponds to the partition
$$
  (2^n)=(\underbrace{2,\ldots,2}_{n}),
$$
i.e., a two‐column rectangle, while the partition $(n,n)$ corresponds to a two‐row rectangle. these two are conjugate to each other. Let $\lambda=(\lambda_1,\ldots,\lambda_s)$ and $\mu=(\mu_1,\ldots,\mu_t)$ be partitions, with the convention that parts beyond the length are zero. We write $\lambda\supseteq\mu$ if $\lambda_i\geq\mu_i$ for all $i\geq1$. The diagram obtained by removing the Young diagram of $\mu$ from that of $\lambda$ is denoted $\lambda/\mu$ and called a {\em skew Young diagram}. If a skew Young diagram has at most one square in each column, it is called a \textit{horizontal strip}. Let $G$ be a finite group and $V$ a finite‐dimensional vector space over $\mathbb{C}$. A mapping
\[
  \rho:G\longrightarrow \mathrm{GL}(V)
\]
such that $\rho(gh)=\rho(g)\rho(h)$ and $\rho(e)=I_V$ is called a {\em representation} of $G$ on $V$, or simply a representation of $G$. Here $\mathrm{GL}(V)$ denotes the group of all invertible linear transformations on $V$, $e$ is the identity element of $G$, and $I_V$ is the identity transformation on $V$. The {\em character} of the representation $\rho$ is the function on $G$ defined by
\[
  \chi_\rho(g)=\operatorname{tr}(\rho(g)),\qquad g\in G.
\]
If there is no subspace of $V$ other than $\{0\}$ and $V$ that is invariant under all linear transformations $\rho(g)$, then $\rho$ is called an {\em irreducible representation}. The character of an irreducible representation is called an {\em irreducible character}.

Let $S_m$ denote the symmetric group. Let $\lambda$ be a partition of $m$, and $\chi_\lambda$ an irreducible character of $S_m$. For an $m\times m$ matrix $B=[b_{ij}]$, the {\em immanant} of $B$ associated with $\chi_\lambda$ is defined by
\begin{eqnarray*}
  \operatorname{Imm}_\lambda(B)
  =\sum_{\pi\in S_m}\chi_{\lambda}(\pi)
    \prod_{i=1}^m b_{i\pi(i)}.
\end{eqnarray*}
When $\lambda=(1^m)$, $\chi_{\lambda}$ is the sign character, and $\operatorname{Imm}_\lambda(\cdot)$ is the determinant. When $\lambda=(m)$, $\chi_{\lambda}$ is the trivial character, and $\operatorname{Imm}_\lambda(\cdot)$ is the permanent.

Immanants have been extensively studied in many areas \cite{gou1,gou2,hai,lam,li,ste}, especially in graph theory; see \cite{chan,don1,don2,mer2,mer3,mer4,wu1,yu}. Valiant~\cite{val} proved that the permanent $p$-family $(\operatorname{per})$ is $\mathrm{VNP}$-complete under $p$-projections. Bürgisser~\cite{Bur2} further studied the computational complexity of immanants and showed that immanant families associated with hook Young diagrams and rectangular Young diagrams of polynomially bounded width are $\mathrm{VNP}$-complete. Mertens and Moore~\cite{mermoore} studied the complexity of the fermionant and constant‐width immanants. Curticapean~\cite{curticapean} gave a dichotomy result for immanant families. Bürgisser~\cite{Bur2} raised the following open problem concerning the more restrictive $p$-projections.

\begin{problem}(B\"urgisser, \cite{Bur2})\label{prob1}
Is the immanant family corresponding to rectangular Young diagrams of width $2$ $\mathrm{VNP}$-complete under $p$-projections?
\end{problem}

de Rugy-Altherre~\cite{rug} studied Problem \ref{prob1} and obtained the following result.

\begin{theorem}
For each positive integer $n$, the family $(\operatorname{Imm}_{(2^n)})_{n\geq 1}$ is $\mathrm{VNP}$-complete under $c$-reductions.
\end{theorem}

In this paper, we give a solution to Problem \ref{prob1}. Our main result is as follows.

\begin{theorem}\label{thm1}
Let $\mathbb{F}$ be a field of characteristic zero, and let $n$ be a positive integer. Define
\[
  C_n(X)=\operatorname{Imm}_{(2^n)}(X),
  \qquad X\in\mathbb{F}^{2n\times 2n}.
\]
Then the polynomial family $(C_n)_{n\geq 1}$ is $\mathrm{VNP}$-complete under $p$-projections.

Furthermore, there exists $N(n)=O(n^2)$ and a $p$-projection $\Pi_n$ such that for any $A\in\mathbb{F}^{n\times n}$,
\begin{eqnarray*}
  \operatorname{Imm}_{(2^{N(n)})}(\Pi_n(A))=\operatorname{per}(A).
\end{eqnarray*}
\end{theorem}

We also prove that the immanant family corresponding to rectangular Young diagrams of length $2$ is $\mathrm{VNP}$-complete under $p$-projections.

\begin{theorem}\label{thm2}
Let $\mathbb{F}$ be a field of characteristic zero, and let $n$ be a positive integer. Define
\[
  R_n(X)=\operatorname{Imm}_{(n,n)}(X),
  \qquad X\in\mathbb{F}^{2n\times 2n}.
\]
Then the polynomial family $(R_n)_{n\geq 1}$ is $\mathrm{VNP}$-complete under $p$-projections.

Moreover, for each $n$, there exists a matrix $U_n\in\mathbb{F}^{n\times n}$ such that for any $A\in\mathbb{F}^{n\times n}$,
\begin{eqnarray*}
  \operatorname{Imm}_{(n,n)}(A\oplus U_n)=\operatorname{per}(A).
\end{eqnarray*}
\end{theorem}

The paper is organized as follows. In Section 2, we give some definitions and lemmas. In Section 3, we prove Theorems \ref{thm1} and \ref{thm2}. Section 4 gives a brief  summary.

\section{Basic preliminaries}

In this section, we give some definitions and lemmas. For undefined notation and terminology, see \cite{jam} and \cite{sag}.

If a finite group $G$ acts on a finite set $\Omega$, i.e., for each $g\in G$ and $\omega\in\Omega$ there is specified an element $g\omega\in\Omega$ such that $e\omega=\omega$ and $(gh)\omega=g(h\omega)$, then on the vector space with basis $\{e_\omega:\omega\in\Omega\}$ we define
\[
  \rho_\Omega(g)e_\omega=e_{g\omega}.
\]
The representation thus obtained is called the {\em permutation representation} of $G$ on $\Omega$. For any $g\in G$, the character value of this permutation representation equals the number of elements fixed by $g$, i.e.,
\[
  \chi_{\rho_\Omega}(g)
  =\#\{\omega\in\Omega:g\omega=\omega\}.
\]

Let $H(z)=\sum_{j=0}^{d}h_jz^j\in \mathbb{F}[z]$, with $h_d\ne0$. If $h_j=h_{d-j}$ for every $0\leq j\leq d$, then we say that the coefficients of $H$ are symmetric with respect to degree $d$, and call $H$ a {\em palindromic polynomial}. That is, the coefficients of $z^j$ and $z^{d-j}$ are always equal. This condition is equivalent to $H(z)=z^dH(z^{-1})$. Now let $\alpha\in\mathbb{F}$, and define the $\alpha$-permanent of an $m\times m$ matrix $X=[x_{ij}]$ as
\begin{eqnarray*}
  \operatorname{Per}_{\alpha,m}(X)
  :=\sum_{\pi\in S_m}\alpha^{c(\pi)}
  \prod_{i=1}^{m}x_{i\pi(i)},
\end{eqnarray*}
where $c(\pi)$ denotes the number of cycles in the disjoint cycle decomposition of $\pi$. Fixed points of length $1$ are also counted as cycles. When the size of the matrix is clear from the context, we abbreviate $\operatorname{Per}_{\alpha,m}(X)$ to $\operatorname{Per}_\alpha(X)$.
By the usual convention, the second–order fermionant and the $(-2)$-permanent satisfy
\begin{eqnarray*}
  \operatorname{Ferm}_{2,m}(X)=(-1)^m\operatorname{Per}_{-2,m}(X).
\end{eqnarray*}
For any $0\leq k\leq m$, define
\begin{eqnarray*}
  a_k^{(m)}(\pi)
  :=\#\{S\subseteq[m]: |S|=k,\ \pi(S)=S\}.
\end{eqnarray*}
Here, $\#$ and $|\cdot|$ both denote cardinality of finite sets, and $\pi(S)=\{\pi(i):i\in S\}$. Thus $a_k^{(m)}(\pi)$ equals the number of $k$-element subsets of $[m]$ that are invariant under $\pi$. From the representation–theoretic viewpoint, it is also the character value of the permutation representation of $S_m$ on the set of all $k$-element subsets of $[m]$ arising from the natural action.

If $H$ is a subgroup of a finite group $G$, and $V$ is a representation of $G$, then $\operatorname{Res}_H^G V$ denotes the {\em restricted representation} obtained by restricting the action of $V$ to $H$. If $W$ is a representation of $H$, then $\operatorname{Ind}_H^G W$ denotes the {\em induced representation} from $W$ to $G$. The notations $\operatorname{Res}$ and $\operatorname{Ind}$ are also used for the corresponding characters. We write $\mathbf{1}_H$ for the one–dimensional trivial representation or trivial character of $H$. When no confusion arises, we simply write $\mathbf{1}$. Moreover, $\cong$ denotes isomorphism of representations, and $\bigoplus$ denotes direct sum of representations.
For a partition $\lambda\vdash m$, let $S^\lambda$ be the complex Specht module of $S_m$. It is an irreducible representation, and its character is denoted by $\chi_\lambda$. When embedding $S_r\times S_s$ into $S_{r+s}$, we agree that the first factor acts on $\{1,\ldots,r\}$ and the second on $\{r+1,\ldots,r+s\}$; this subgroup is called a Young subgroup. If $U$ and $V$ are representations of $S_r$ and $S_s$, respectively, then $U\boxtimes V$ denotes their outer tensor product on $S_r\times S_s$. Its character is denoted by $\chi_U\otimes\chi_V$ and satisfies
\[
  (\chi_U\otimes\chi_V)(\sigma,\tau)
  =\chi_U(\sigma)\chi_V(\tau).
\]
Let $s_\lambda$ be the Schur symmetric function corresponding to the partition $\lambda$, and write $h_t=s_{(t)}$ for the $t$-th complete homogeneous symmetric function. The Littlewood–Richardson coefficients $c_{\mu,\nu}^{\lambda}$ are the structure constants in the product of Schur functions:
\[
  s_\mu s_\nu=\sum_\lambda c_{\mu,\nu}^{\lambda}s_\lambda.
\]
We denote the Frobenius character map by $\operatorname{ch}$. This map sends the character of $S^\lambda$ to $s_\lambda$, direct sums of representations to sums of symmetric functions, and the outer tensor product induced from a Young subgroup to the product of the corresponding symmetric functions. Pieri's rule is a special case of the Littlewood–Richardson rule: for any partition $\mu$ and positive integer $k$,
\[
  s_\mu h_k=\sum_\lambda s_\lambda,
\]
where the sum is over all partitions $\lambda$ such that $\lambda/\mu$ is a horizontal strip of size $k$, and each admissible $s_\lambda$ has coefficient $1$.
For any finite set $T$, denote by $\mathfrak S(T)$ the permutation group of all bijections on $T$. For an $m\times m$ matrix $X=[x_{ij}]$ and a subset $S\subseteq[m]$, write $X[S,S]$ for the {\em principal submatrix} obtained by taking the rows and columns whose indices belong to $S$, and write $S^{\mathrm{c}}=[m]\setminus S$ for the complement of $S$ in $[m]$. For a proposition $E$, its indicator function is denoted by $\mathbf{1}_E$: it equals $1$ if $E$ is true, and $0$ otherwise. For any $0\leq k\leq m$, define
\begin{eqnarray*}
  P_k(X)
  &:=&\sum_{\substack{S\subseteq[m]\\|S|=k}}
      \operatorname{per} X[S,S]~\operatorname{per} X[S^{\mathrm{c}},S^{\mathrm{c}}],\\
  D_k(X)
  &:=&\sum_{\substack{S\subseteq[m]\\|S|=k}}
      \operatorname{det} X[S,S]~\operatorname{det} X[S^{\mathrm{c}},S^{\mathrm{c}}],
\end{eqnarray*}
where $P_k(X)$ and $D_k(X)$ are called the permanent principal minor convolution and the determinant principal minor convolution, respectively. When $S=\varnothing$, $X[S,S]$ is the $0\times0$ empty matrix. By convention, the permanent and determinant of the empty matrix are both $1$. The {\em Leibniz expansion} is the weighted sum of permutation monomials by the sign of the permutation.

A {\em weighted directed graph} is a finite directed graph in which each directed edge is assigned a weight in a commutative $\mathbb{F}$-algebra. Edges not present are regarded as edges of weight $0$.
We use a fixed six–vertex directed subgraph to reduce the permanent to the $(-2)$-permanent. After ordering the vertices of the graph, its adjacency matrix is denoted by $M=[m_{uv}]$, where $m_{uv}$ is the weight of the directed edge $u\to v$. A {\em directed cycle cover} is a union of pairwise vertex–disjoint directed cycles such that every vertex belongs to exactly one cycle. Equivalently, in the chosen edges, each vertex has exactly one incoming and one outgoing edge. The weight of a cycle cover is the product of the weights of all its edges. If this product is nonzero, the cover is called a nonzero cycle cover.

Let
\begin{eqnarray*}
  W=
  \begin{pmatrix}
    -1&\frac12&1\\
    -1&\frac12&1\\
    1&-\frac12&-\frac12
  \end{pmatrix}.
\end{eqnarray*}
Given an $n\times n$ generic variable matrix $A=(a_{ij})$, we construct the following weighted directed graph over the polynomial ring $\mathbb{F}[a_{ij}:1\leq i,j\leq n]$.
First take vertices
\[
  L_1,\ldots,L_n,\qquad R_1,\ldots,R_n.
\]
For each ordered pair $e=(i,j)\in[n]^2$, add six vertices
$u_{e,1},u_{e,2},u_{e,3},v_{e,1},v_{e,2},v_{e,3}$, and set directed edges and their weights as follows:
\begin{eqnarray*}
  u_{e,r}&\longrightarrow&v_{e,r}\qquad \text{weight }1,\quad r=1,2,3,\\
  v_{e,r}&\longrightarrow&u_{e,s}\qquad \text{weight }W_{rs},\quad r,s=1,2,3,\\
  L_i&\longrightarrow&u_{e,1}\qquad \text{weight }a_{ij},\\
  v_{e,1}&\longrightarrow&R_j\qquad \text{weight }1,\\
  R_j&\longrightarrow&u_{e,2}\qquad \text{weight }1,\\
  v_{e,2}&\longrightarrow&L_i\qquad \text{weight }1.
\end{eqnarray*}
All other edges have weight zero. For each $e=(i,j)$, the six vertices
$u_{e,1},u_{e,2},u_{e,3},v_{e,1},v_{e,2},v_{e,3}$ together with the edges among them are called the six–vertex subgraph corresponding to $e$. The four edges connecting $L_i$, $R_j$ to this six–vertex subgraph are called external edges. Among them, the two edges entering the six–vertex subgraph are called external incoming edges, and the two edges leaving it are called external outgoing edges. Let the adjacency matrix of the whole directed graph be $M_n(A)$; its size is
$m(n)=2n+6n^2$. By construction, every nonzero cycle cover must contain all edges
$u_{e,r}\longrightarrow v_{e,r}$. Now fix the six–vertex subgraph corresponding to $e=(i,j)$. For a given cycle cover, let
$I\subseteq\{1,2\}$ be the set of labels of the selected external incoming edges, and $O\subseteq\{1,2\}$ the set of labels of the selected external outgoing edges. Specifically, if the external incoming edge $L_i\to u_{e,1}$ is chosen, we put label $1$ into $I$;
if $R_j\to u_{e,2}$ is chosen, we put label $2$ into $I$. Similarly, if the external outgoing edge $v_{e,1}\to R_j$ is chosen, put $1$ into $O$; if $v_{e,2}\to L_i$ is chosen, put $2$ into $O$.
If $s\in I$, then $u_{e,s}$ already has an incoming edge from outside the six–vertex subgraph.
If $r\in O$, then the outgoing edge of $v_{e,r}$ already points outside. Therefore, the number of $u$ vertices that still need an incoming edge from inside the subgraph is $3-|I|$, and the number of $v$ vertices that still need an outgoing edge inside the subgraph is $3-|O|$. Only when $|I|=|O|$ can these $v$ vertices be paired with the $u$ vertices using internal edges.
For an $m\times m$ matrix $X$ and an indeterminate $z$, define the following principal minor generating polynomial:
\begin{eqnarray*}
  \Phi_X(z)
  :=\sum_{k=0}^{m}D_k(X)z^k
  =\sum_{S\subseteq[m]}
    \operatorname{det} X[S,S]\operatorname{det} X[S^{\mathrm{c}},S^{\mathrm{c}}]z^{|S|}.
\end{eqnarray*}
In what follows, equalities involving $z^{-1}$ are understood in the Laurent polynomial ring $\mathbb{F}[z,z^{-1}]$. Laurent polynomials are allowed to have finitely many negative powers. If $F(z)=\sum_jf_jz^j$, write $[z^j]F(z)=f_j$, i.e., $[z^j]F(z)$ denotes the coefficient of $z^j$ in $F$. Thus $[z^a]-[z^b]$ denotes the linear operator ``extract the coefficient of $z^a$ and subtract the coefficient of $z^b$''. For a nonzero polynomial $F$, write $\deg F$ for the degree of its highest nonzero term.

A square matrix is called a {\em monomial matrix} if each row and each column contains exactly one nonzero entry. For a monomial matrix $P=[p_{ij}]$, the column containing the unique nonzero entry in row $i$ is denoted by $\rho(i)$. Since each column also has exactly one nonzero entry, the map $i\mapsto\rho(i)$ is a permutation, uniquely determined by $P$; we call $\rho$ the {\em associated permutation} of the monomial matrix $P$. For a positive integer $q$, set $L=2^q-1$. Partition the set $[L]$ into consecutive intervals of lengths
$1,2,4,\ldots,2^{q-1}$, and on each interval take the cyclic permutation consisting of all elements of that interval. Since
\[
  1+2+4+\cdots+2^{q-1}=2^q-1=L,
\]
these disjoint cycles together determine a permutation $\rho_L$ on $[L]$. For each $\varepsilon\in\{1,-1\}$, define the monomial matrix $P_L^{[\varepsilon]}$ as follows: its associated permutation is $\rho_L$, and
\[
  (P_L^{[\varepsilon]})_{i,\rho_L(i)}=\omega_i,
  \qquad
  \omega_1=\varepsilon\,\operatorname{sgn}(\rho_L),\quad
  \omega_i=1\ (2\leq i\leq L),
\]
and all other entries are zero. For a positive integer $m$, fix
\[
  q=q(m):=\min\{q\geq1:2^q-1\geq2m+7\},
  \qquad L=L(m):=2^{q(m)}-1,
\]
and write
\[
  S_L(z)=1+z+\cdots+z^L.
\]
Define
\begin{eqnarray}\label{equ9}
  K_m(z)=
  \begin{cases}
    S_L(z)^3,&m~\text{is odd},\\[1mm]
    \dfrac{1+z}{2}S_L(z)^3,&m~\text{is even}.
  \end{cases}
\end{eqnarray}
Let $r=r(m):=\deg K_m$, and set $N=N(m):=(m+r(m))/2$. A matrix whose entries are constants (i.e., elements of $\mathbb{F}$ containing no variables) is called a constant matrix. For the above positive integer $m$ and the uniquely determined $q(m)$ and $L(m)$, define the constant matrix
\begin{eqnarray*}
  \mathcal B_m=
  \begin{cases}
    P_L^{[-1]}\oplus P_L^{[+1]}\oplus P_L^{[+1]},
      &m\text{ odd},\\[1mm]
    P_L^{[+1]}\oplus P_L^{[+1]}\oplus P_L^{[+1]}\oplus[\frac12],
      &m\text{ even}.
  \end{cases}
\end{eqnarray*}
where $[1/2]$ denotes the $1\times1$ matrix whose sole entry is $1/2$, and $\dim\mathcal B_m$ denotes the size of the square matrix $\mathcal B_m$. From the sizes of the blocks, we have $\dim\mathcal B_m=r(m)=\deg K_m$. For a positive integer $n$, define the constant matrix:
\begin{eqnarray*}
  U_n=
  \begin{pmatrix}
    \dfrac1{n!}&\dfrac1{n!}&\cdots&\dfrac1{n!}\\
    1&1&\cdots&1\\
    \vdots&\vdots&&\vdots\\
    1&1&\cdots&1
  \end{pmatrix}\in\mathbb{F}^{n\times n}.
\end{eqnarray*}
Since $\mathbb{F}$ has characteristic zero, $n!$ is nonzero and invertible in $\mathbb{F}$, so this constant matrix is well–defined. For two complex–valued class functions $\varphi$ and $\psi$ on $S_n$, their standard inner product is
\[
  \langle\varphi,\psi\rangle_{S_n}
  :=\frac1{n!}\sum_{\tau\in S_n}
  \varphi(\tau)\overline{\psi(\tau)}.
\]
Here $\overline{\psi(\tau)}$ denotes the complex conjugate of $\psi(\tau)$. Although the base field throughout is an arbitrary field $\mathbb{F}$ of characteristic zero, the irreducible characters of the symmetric group are integer–valued. Hence the orthogonality relations below may be used first over $\mathbb{C}$; the resulting values $0$ or $1$ are then interpreted as identities in $\mathbb{F}$ via the natural embedding $\mathbb{Q}\hookrightarrow\mathbb{F}$. Every permutation monomial of $U_n$ takes $1/n!$ from the first row and $1$ from each of the remaining rows, so each permutation monomial has weight $1/n!$. Therefore, for any partition $\nu\vdash n$, we have
\begin{equation}\label{equ5}
  \operatorname{Imm}_\nu(U_n)
  =\frac1{n!}\sum_{\tau\in S_n}\chi_{\nu}(\tau)
  =\langle\chi_{\nu},\mathbf{1}\rangle_{S_n}
  =\begin{cases}
    1,&\nu=(n),\\
    0,&\nu\ne(n).
  \end{cases}
\end{equation}
The first equality follows from the definition of immanant and the fact that all permutation monomials have the same weight. The second equality comes from the definition of the inner product of class functions. The last equality uses the orthogonality of irreducible characters: the inner product of $\chi_\nu$ with the trivial character $\mathbf{1}$ is $1$ only when $\chi_\nu=\mathbf{1}$, i.e., $\nu=(n)$, and is $0$ otherwise.
Let $\operatorname{per}_n$ denote the permanent polynomial corresponding to an $n\times n$ generic variable matrix.

\begin{lemma}(James and Kerber, \cite{jam})\label{lem2.01}
Let $n$ be a positive integer and $\lambda$ a partition of $n$. Then for every $\pi\in S_n$,
\begin{eqnarray*}
\chi_{\lambda^\top}(\pi) = \operatorname{sgn}(\pi) \, \chi_{\lambda}(\pi),
\end{eqnarray*}
where $\operatorname{sgn}(\pi)$ denotes the sign of $\pi$: $+1$ for even permutations and $-1$ for odd permutations.
\end{lemma}

\begin{lemma}\label{lem2.1}
Let $m$ be a positive integer and $k$ an integer with $0\leq k\leq m/2$. Then
\begin{eqnarray*}
  \operatorname{Ind}_{S_{m-k}\times S_k}^{S_m}\mathbf{1}
  \cong\bigoplus_{j=0}^{k}S^{(m-j,j)}.
\end{eqnarray*}
Moreover,
\begin{eqnarray*}
  a_k^{(m)}=\sum_{j=0}^{k}\chi_{(m-j,j)}.
\end{eqnarray*}
\end{lemma}

\begin{proof}
Write
\[
  V=\operatorname{Ind}_{S_{m-k}\times S_k}^{S_m}\mathbf{1}.
\]
Here $\mathbf{1}$ is the one–dimensional trivial representation of the subgroup $S_{m-k}\times S_k$. We first note that $V$ is exactly the permutation representation of $S_m$ on the set of all $k$-element subsets of $[m]$.
Take the standard $k$-element subset
\[
  T_0=\{m-k+1,m-k+2,\ldots,m\}.
\]
Under the natural action of $S_m$, the permutations that stabilise $T_0$ may independently permute the $m-k$ elements of $[m]\setminus T_0$ and the $k$ elements of $T_0$.
Thus the stabiliser of $T_0$ is precisely $S_{m-k}\times S_k$. Here, the stabiliser of an object is the subgroup of all elements of the group that fix that object. For a subgroup $H$ of $S_m$, the left coset $\sigma H$ denotes the set $\{\sigma h:h\in H\}$. It follows that the coset space $S_m/(S_{m-k}\times S_k)$ is in bijection with the set of all $k$-element subsets of $[m]$: the coset $\sigma(S_{m-k}\times S_k)$ corresponds to the subset $\sigma(T_0)$. Hence the induced representation $V$ from the trivial representation is exactly the permutation representation on the basis of all $k$-element subsets. The character value of a permutation representation at $\pi\in S_m$ equals the number of basis vectors fixed by $\pi$. Here each basis vector corresponds to a $k$-element subset $S$, and it is fixed precisely when $\pi(S)=S$. Let $\chi_V$ be the character of $V$. By the definition of $a_k^{(m)}$,
\[
  \chi_V(\pi)=a_k^{(m)}(\pi).
\]
Under the Frobenius character map, the trivial representation of the symmetric group corresponds to the one–row Schur function, i.e.,
\[
  \operatorname{ch}(\mathbf{1}_{S_{m-k}})=h_{m-k}=s_{(m-k)},
  \qquad
  \operatorname{ch}(\mathbf{1}_{S_k})=h_k=s_{(k)}.
\]
The trivial representation of $S_{m-k}\times S_k$ can be written as $\mathbf{1}_{S_{m-k}}\boxtimes\mathbf{1}_{S_k}$. Since induction of an outer tensor product corresponds under the Frobenius character map to multiplication of symmetric functions,
\[
  \operatorname{ch}(V)=h_{m-k}h_k=s_{(m-k)}h_k.
\]
By Pieri's rule, the Schur function expansion of $s_{(m-k)}h_k$ is obtained by adding a horizontal strip of size $k$ to the one–row Young diagram $(m-k)$. If $j$ of the new boxes lie in the second row, the remaining $k-j$ lie in the first row, yielding the partition $((m-k)+(k-j),j)=(m-j,j)$. Since $0\leq j\leq k$ and $k\leq m/2$, we have $m-j\geq j$, so these pairs are all partitions. Because $j\leq k\leq m-k$, the new boxes added to the second row occupy the first $j$ columns, which do not overlap with the new columns extending the first row; this satisfies the horizontal–strip condition ``at most one new box per column''. Conversely, starting from the one–row diagram $(m-k)$ and adding a horizontal strip of size $k$ cannot produce a third row: if a box appeared in the third row, the same column would also have a box in the second row, giving at least two new boxes in that column, contradicting the horizontal–strip condition. Hence the number of new boxes in the second row can only be $j=0,1,\ldots,k$, and no other shapes occur. Pieri's rule also guarantees that each admissible shape has coefficient $1$. Thus
\[
  \operatorname{ch}(V)=\sum_{j=0}^{k}s_{(m-j,j)}.
\]
The Frobenius character map sends the irreducible representation $S^\lambda$ to $s_\lambda$ and preserves direct sums, so
\[
  V\cong\bigoplus_{j=0}^{k}S^{(m-j,j)}.
\]
Taking characters of both sides of this isomorphism, and using $\chi_V=a_k^{(m)}$, yields
\[
  a_k^{(m)}=\sum_{j=0}^{k}\chi_{(m-j,j)}.
\]
\end{proof}

Applying Lemma \ref{lem2.1} with $m=2n$ and $k=n$ and $k=n-1$, respectively, gives
\begin{eqnarray}\label{equ3}
  \chi_{{(n,n)}}=a_n^{(2n)}-a_{n-1}^{(2n)}.
\end{eqnarray}
Since $(n,n)^\top=(2^n)$, it follows from Lemma \ref{lem2.01} that
\begin{eqnarray}\label{equ4}
  \chi_{(2^n)}(\pi)
  =\operatorname{sgn}(\pi)\bigl(a_n^{(2n)}(\pi)-a_{n-1}^{(2n)}(\pi)\bigr).
\end{eqnarray}

\begin{lemma}\label{lem2.2}
Let $n$ be a positive integer and $X=(x_{ij})_{1\leq i,j\leq2n}$ a generic variable matrix. Then each of the following holds.

$(i)$ $\operatorname{Imm}_{(n,n)}(X)=P_n(X)-P_{n-1}(X)$.

$(ii)$  $\operatorname{Imm}_{(2^n)}(X)=D_n(X)-D_{n-1}(X)$.
\end{lemma}

\begin{proof}
$(i)$ First fix any $k\in\{0,1,\ldots,2n\}$. By definition of $a_k^{(2n)}(\pi)$, it equals the number of $k$-element subsets fixed by the permutation $\pi$, hence can be written as the sum of indicator functions:
\[
  a_k^{(2n)}(\pi)
  =\sum_{\substack{S\subseteq[2n]\\|S|=k}}
    \mathbf{1}_{\{\pi(S)=S\}}.
\]
Substituting this into the following sum and interchanging the two finite sums gives
\begin{eqnarray*}
  \sum_{\pi\in S_{2n}}a_k^{(2n)}(\pi)
    \prod_i x_{i\pi(i)}
  &=&\sum_{\substack{S\subseteq[2n]\\|S|=k}}
    \sum_{\substack{\pi\in S_{2n}\\\pi(S)=S}}
    \prod_i x_{i\pi(i)}  \\
  &=&\sum_{|S|=k}\operatorname{per} X[S,S]~\operatorname{per} X[S^{\mathrm{c}},S^{\mathrm{c}}]  \\
  &=&P_k(X).
\end{eqnarray*}
For a fixed $S$, the condition $\pi(S)=S$ also implies $\pi(S^{\mathrm{c}})=S^{\mathrm{c}}$. Let $\pi|_S:S\to S$ be the restriction of $\pi$ to $S$; then $\pi$ decomposes uniquely as a permutation $\sigma=\pi|_S$ on $S$ and a permutation $\tau=\pi|_{S^{\mathrm{c}}}$ on $S^{\mathrm{c}}$. The corresponding permutation monomial factorises as
\begin{eqnarray}\label{equa2.3}
  \prod_{i=1}^{2n}x_{i\pi(i)}
  =\left(\prod_{i\in S}x_{i\sigma(i)}\right)
   \left(\prod_{i\in S^{\mathrm{c}}}x_{i\tau(i)}\right).
\end{eqnarray}
Combining \eqref{equa2.3} and summing over $\sigma$ and $\tau$ separately, the two brackets give the permanents of the principal submatrices $X[S,S]$ and $X[S^{\mathrm{c}},S^{\mathrm{c}}]$, respectively, yielding the second equality above. From \eqref{equ3}, $\chi_{(n,n)}=a_n^{(2n)}-a_{n-1}^{(2n)}$. Taking $k=n$ and $k=n-1$ gives $\operatorname{Imm}_{(n,n)}(X)=P_n(X)-P_{n-1}(X)$.

$(ii)$ Again fix $k$. If $\pi(S)=S$, then $\pi$ is the product of two permutations $\pi|_S$ and $\pi|_{S^{\mathrm{c}}}$ acting on disjoint sets. The sign of a permutation is multiplicative. For permutations on finite subsets, we order the elements increasingly and compute the sign; this only gives a conjugate permutation and does not change the sign. Hence
\[
  \operatorname{sgn}(\pi)
  =\operatorname{sgn}(\pi|_S)\operatorname{sgn}(\pi|_{S^{\mathrm{c}}}).
\]
Similar to (i), we obtain
\begin{eqnarray*}
  \sum_{\pi\in S_{2n}}\operatorname{sgn}(\pi)a_k^{(2n)}(\pi)
  \prod_i x_{i\pi(i)}
  &=&\sum_{\substack{S\subseteq[2n]\\|S|=k}}
    \left(\sum_{\sigma\in \mathfrak S(S)}\operatorname{sgn}(\sigma)
      \prod_{i\in S}x_{i\sigma(i)}\right)\\
  &&{}\cdot
    \left(\sum_{\tau\in \mathfrak S(S^{\mathrm{c}})}\operatorname{sgn}(\tau)
      \prod_{i\in S^{\mathrm{c}}}x_{i\tau(i)}\right)\\
  &=&\sum_{|S|=k}\operatorname{det} X[S,S]~\operatorname{det} X[S^{\mathrm{c}},S^{\mathrm{c}}]  \\
  &=&D_k(X).
\end{eqnarray*}
Here the two brackets are the Leibniz determinant expansions of the corresponding principal submatrices. Finally, from \eqref{equ4}, $\chi_{(2^n)}(\pi)=\operatorname{sgn}(\pi)\bigl(a_n^{(2n)}(\pi)-a_{n-1}^{(2n)}(\pi)\bigr)$; taking $k=n$ and $k=n-1$ gives $\operatorname{Imm}_{(2^n)}(X)=D_n(X)-D_{n-1}(X)$.
\end{proof}

\begin{lemma}\label{lem2.3}
Let $n$ be a positive integer. Then the polynomial families $\bigl(\operatorname{Imm}_{(n,n)}\bigr)_{n\geq1}$
and $\bigl(\operatorname{Imm}_{(2^n)}\bigr)_{n\geq1}$ belong to $\mathrm{VNP}$.
\end{lemma}

\begin{proof}
Fix an integer–valued $p$-bounded function $k=k(m)$ with $0\leq k(m)\leq m$ for all $m$, and write down the Boolean summation formulas that show the polynomial families
$\bigl(D_{k(m)}\bigr)_{m\geq1}$ and $\bigl(P_{k(m)}\bigr)_{m\geq1}$ satisfy the definition of VNP. The variables used in the sums are restricted to $0$ or $1$, their number is bounded by a polynomial in $m$, and the number of arithmetic operations needed to compute each summand is also polynomially bounded in $m$. Finally take $m=2n$ and set $k=n$ or $k=n-1$ to obtain the two immanant families in the lemma.

Consider $D_k$. For each $0\leq r\leq m$, define the Lagrange interpolation polynomial
\begin{eqnarray*}
  \delta_r^{(m)}(t)
  :=\prod_{\substack{0\leq q\leq m\\q\ne r}}
    \frac{t-q}{r-q}.
\end{eqnarray*}
When $t=r$, every factor in the product equals $1$. When $t\in\{0,1,\ldots,m\}$ and $t\ne r$, the factor with $q=t$ is $0$. Therefore
\[
  \delta_r^{(m)}(t)=
  \begin{cases}
    1,&t=r,\\
    0,&t\in\{0,1,\ldots,m\}\setminus\{r\}.
  \end{cases}
\]
Each denominator $r-q$ is a nonzero integer. Since $\mathbb{F}$ has characteristic zero, these integers are nonzero in $\mathbb{F}$, hence invertible. Thus $\delta_r^{(m)}$ is a polynomial in $\mathbb{F}[t]$.
For each nonnegative integer $t$, let $I_t$ be the $t\times t$ identity matrix, i.e., all diagonal entries are $1$ and off–diagonal entries are $0$. In particular, $I_0$ denotes the empty matrix.
For $s=(s_1,\ldots,s_m)\in\{0,1\}^m$, set $E_s=\operatorname{diag}(s_1,\ldots,s_m)$, the diagonal matrix with diagonal entries $s_1,\ldots,s_m$. If $S=\{i\in[m]:s_i=1\}$, then
\begin{eqnarray}\label{equa2.1}
  \operatorname{det}(I_m-E_s+E_sXE_s)&=&\operatorname{det} X[S,S]
  \end{eqnarray}
  and
  \begin{eqnarray}\label{equa2.2}
  \operatorname{det}(E_s+(I_m-E_s)X(I_m-E_s))&=&\operatorname{det} X[S^{\mathrm{c}},S^{\mathrm{c}}].
\end{eqnarray}
Indeed, after permuting rows and columns according to the order $S,S^{\mathrm{c}}$, the first matrix becomes $X[S,S]\oplus I_{|S^{\mathrm{c}}|}$, and the second becomes $I_{|S|}\oplus X[S^{\mathrm{c}},S^{\mathrm{c}}]$. Since simultaneous row and column permutations do not change the determinant, equalities \eqref{equa2.1} and \eqref{equa2.2} hold. Because $s_i\in\{0,1\}$, we have $\sum_i s_i=|S|$. Hence $\delta_k^{(m)}(\sum_i s_i)$ is $1$ only when $|S|=k$, so only those subsets survive. Therefore
\begin{eqnarray*}
  D_k(X)
  =\sum_{s\in\{0,1\}^m}
  \delta_k^{(m)}\!\left(\sum_i s_i\right)
  \operatorname{det}(I_m-E_s+E_sXE_s)
  \operatorname{det}(E_s+(I_m-E_s)X(I_m-E_s)).
\end{eqnarray*}
The determinant can be computed without division, and the number of additions and multiplications required is polynomially bounded in $m$. The polynomial $\delta_{k(m)}^{(m)}$ contains only $m$ linear factors, and each entry of the above two matrices is a low–degree polynomial in $X$ and $s$. Hence the number of arithmetic operations to compute each summand is still polynomially bounded in $m$. The whole sum involves only $m$ Boolean variables $s_i$, so the formula meets the definition of VNP. This means $\bigl(D_{k(m)}\bigr)_{m\geq1}\in\mathrm{VNP}$.

Now consider $P_{k(m)}$. In addition to the Boolean variables $s_1,\ldots,s_m$, introduce $m^2$ Boolean variables $y_{ij}$, $1\leq i,j\leq m$, and define
\begin{eqnarray*}
  G_{m,k}(X;s,y)
  &:=\;&\delta_k^{(m)}\!\left(\sum_i s_i\right)
  \prod_i\delta_1^{(m)}\!\left(\sum_jy_{ij}\right)
  \prod_j\delta_1^{(m)}\!\left(\sum_iy_{ij}\right)
  \notag\\
  &&\cdot\prod_{i,j}\bigl(1-y_{ij}(s_i-s_j)^2\bigr)
  \prod_{i,j}\bigl(1-y_{ij}+y_{ij}x_{ij}\bigr).
\end{eqnarray*}
We now explain the role of each factor.
\begin{enumerate}
\item The factor $\delta_k^{(m)}(\sum_i s_i)$ keeps only those $s$ with $|S|=k$, where $S=\{i:s_i=1\}$.
\item For row $i$, the factor $\delta_1^{(m)}(\sum_jy_{ij})$ is nonzero exactly when that row has precisely one $y_{ij}=1$. Similarly, for column $j$, the factor $\delta_1^{(m)}(\sum_iy_{ij})$ is nonzero exactly when that column has precisely one $y_{ij}=1$. Thus the second and third products together ensure that $y=[y_{ij}]$ is a permutation matrix. Denote its associated permutation by $\pi$, i.e., $y_{i,\pi(i)}=1$.
\item When $y_{ij}=0$, the factor $1-y_{ij}(s_i-s_j)^2$ is identically $1$. When $y_{ij}=1$, it equals $1$ if $s_i=s_j$ and $0$ otherwise.
      Hence the fourth product requires $s_i=s_{\pi(i)}$ for all $i$, which is equivalent to $\pi(S)=S$.
\item When $y_{ij}=0$, $1-y_{ij}+y_{ij}x_{ij}=1$. When $y_{ij}=1$, this factor equals $x_{ij}$. Thus the last product gives exactly the permutation monomial $\prod_i x_{i,\pi(i)}$.
\end{enumerate}
Thus, for fixed $s$, the nonzero $y$–terms are in bijection with the permutations that stabilise $S$. Summing over all $s$ gives
\begin{eqnarray*}
  P_k(X)
  =\sum_{s\in\{0,1\}^m}
   \sum_{y\in\{0,1\}^{m^2}}G_{m,k}(X;s,y).
\end{eqnarray*}
For each $k$-element subset $S$, expanding $\operatorname{per} X[S,S]\operatorname{per} X[S^{\mathrm{c}},S^{\mathrm{c}}]$ gives exactly one term for each permutation $\pi$ stabilising $S$. The polynomial $G_{m,k}$ consists of $O(m^2)$ factors and interpolation polynomials of degree at most $m$, so the number of additions, subtractions, and multiplications needed to compute it is polynomially bounded in $m$. The total number of Boolean variables used in the sum is $m+m^2$, which is also polynomially bounded in $m$. Hence $\bigl(P_{k(m)}\bigr)_{m\geq1}\in\mathrm{VNP}$.
From Lemma \ref{lem2.2}, we get that $\operatorname{Imm}_{(n,n)}=P_n-P_{n-1}$ and $\operatorname{Imm}_{(2^n)}=D_n-D_{n-1}$.
Let $(F_n)$ and $(G_n)$ be two VNP families, written as
\[
  F_n(x)=\sum_{u\in\{0,1\}^{a(n)}}f_n(x,u),
  \qquad
  G_n(x)=\sum_{v\in\{0,1\}^{b(n)}}g_n(x,v).
\]
Introduce a new Boolean variable $\beta$ and define
\[
  h_n(x;\beta,u,v)
  =(1-\beta)f_n(x,u)\prod_{j=1}^{b(n)}(1-v_j)
  -\beta g_n(x,v)\prod_{i=1}^{a(n)}(1-u_i).
\]
When $\beta=0$, the product $\prod_j(1-v_j)$ is nonzero only for $v=(0,\ldots,0)$, so summing over $u,v$ gives $F_n(x)$.
When $\beta=1$, the product $\prod_i(1-u_i)$ is nonzero only for $u=(0,\ldots,0)$, and the sum gives $-G_n(x)$. Therefore
\[
  F_n(x)-G_n(x)
  =\sum_{\beta\in\{0,1\}}
   \sum_{u\in\{0,1\}^{a(n)}}
   \sum_{v\in\{0,1\}^{b(n)}}h_n(x;\beta,u,v).
\]
If the number of arithmetic operations needed to compute $f_n$ and $g_n$ is polynomially bounded in $n$, and $a(n),b(n)$ are $p$-bounded, then $h_n$ also satisfies the same requirements. Thus the above formula is exactly the Boolean summation required by the definition of VNP, so $(F_n-G_n)_{n\geq1}\in\mathrm{VNP}$.
Therefore both differences above belong to VNP. Finally take $m=2n$ and set $k=n$ and $k=n-1$, respectively. The number of matrix variables is $(2n)^2$, the degree of the immanants is $2n$, and the number of summation variables as well as the arithmetic cost per summand are polynomially bounded in $n$. Hence both sequences are VNP polynomial families.
\end{proof}

\begin{lemma}\label{lem2.6}
Let $n$ be a positive integer, $A=[a_{ij}]$ an $n\times n$ generic variable matrix (i.e., the $a_{ij}$ are independent indeterminates), and set $m(n)=6n^2+2n$. Then there exists an $m(n)\times m(n)$ matrix $\widehat M_n(A)$, each of whose entries is either one of the variables $a_{ij}$ or a constant from $\mathbb{F}$, such that
\begin{eqnarray*}
  \operatorname{Per}_{-2,m(n)}\bigl(\widehat M_n(A)\bigr)=\operatorname{per}(A).
\end{eqnarray*}
Consequently,
$(\operatorname{per}_n)_{n\geq 1}\leq_p(\operatorname{Per}_{-2,m})_{m\geq 1}$.
\end{lemma}

\begin{proof}
For an adjacency matrix $M=[m_{uv}]$ and a permutation $\pi$, the edges corresponding to $\pi$ are $u\to\pi(u)$. Since every vertex chooses exactly one outgoing edge and also exactly one incoming edge, these edges form several vertex–disjoint directed cycles, i.e., a directed cycle cover. The number of cycles is exactly $c(\pi)$, and the weight of the corresponding term is
\[
  \alpha^{c(\pi)}\prod_u m_{u\pi(u)}.
\]
In particular, when $\alpha=-2$, each directed cycle in the cover contributes a factor $-2$. Consider the original matrix $M_n(A)$, with edge weights as defined above. For every $e=(i,j)\in[n]^2$ and $r\in\{1,2,3\}$, the vertex $u_{e,r}$ has only one nonzero–weight outgoing edge, namely $u_{e,r}\longrightarrow v_{e,r}$.
In any nonzero cycle cover, every vertex must choose a nonzero–weight outgoing edge; hence these $3n^2$ edges must appear in every nonzero cycle cover; we shall call them mandatory edges. After choosing these mandatory edges, for each six–vertex subgraph we only need to determine the outgoing edges of the vertices $v_{e,r}$: each may point to some $u_{e,s}$, or, when permitted, to the outside of the six–vertex subgraph. Fix the subgraph corresponding to $e=(i,j)$ and adopt the following convention:
\begin{itemize}
\item $1\in I$ means the external incoming edge $L_i\to u_{e,1}$ is chosen,
      $2\in I$ means the external incoming edge $R_j\to u_{e,2}$ is chosen.
\item $1\in O$ means the external outgoing edge $v_{e,1}\to R_j$ is chosen,
      $2\in O$ means the external outgoing edge $v_{e,2}\to L_i$ is chosen.
\end{itemize}
First fix the choice of all external edges in the whole graph. For each six–vertex subgraph $e$, denote the corresponding label sets by $I_e$ and $O_e$. If $s\in I_e$, then $u_{e,s}$ already has an external incoming edge, so it cannot be pointed to by any $v_{e,r}$. If $r\in O_e$, then the outgoing edge of $v_{e,r}$ already leaves the subgraph, so it cannot connect to any $u_{e,s}$. Thus, the number of $v$ vertices that still need an outgoing edge inside the subgraph is $3-|O_e|$, and the number of $u$ vertices that still need an incoming edge is $3-|I_e|$. If $|I_e|\ne|O_e|$, the numbers differ and a cycle cover cannot be completed, so the contribution of this case is $0$.
When $|I_e|=|O_e|$, an internal completion of this six–vertex subgraph is uniquely represented by a bijection
\[
  \eta_e:\{1,2,3\}\setminus O_e
  \longrightarrow\{1,2,3\}\setminus I_e,
\]
where $\eta_e(r)=s$ means we choose the internal edge $v_{e,r}\to u_{e,s}$. Write
\[
  w_e(\eta_e):=\prod_{r\notin O_e}W_{r,\eta_e(r)}.
\]
The mandatory edges together with the internal edges chosen by $\eta_e$ form in this six–vertex subgraph two kinds of vertex–disjoint directed components: those entirely contained in the six–vertex subgraph (directed cycles), and those that start with an external incoming edge and end with an external outgoing edge (directed paths). Let $\ell_e(\eta_e)$ be the number of directed cycles of the first kind.
The correspondence between external incoming and outgoing edges determined by these directed paths does not depend on $\eta_e$. Indeed, when $|I_e|=|O_e|=0$ there are no such paths; when $|I_e|=|O_e|=1$ there is a unique label for each; when $|I_e|=|O_e|=2$ we have $I_e=O_e=\{1,2\}$, and the two paths connect the mandatory edges of label $1$ and $2$ to the external outgoing edges of the same label. Contract each such directed path together with its two external edges into a single directed edge connecting the external vertices, and temporarily delete the directed cycles completely contained inside the six–vertex subgraph. Doing this for all subgraphs yields a directed graph $\Gamma_0$ that depends only on the already fixed external edges, not on the specific choice of the internal completions $\eta_e$.
If the fixed external edges do not give each $L_i$ and $R_j$ exactly one incoming and one outgoing edge, then no cycle cover exists. Otherwise, let $c_0$ be the number of directed cycles in $\Gamma_0$, and let $w_0$ be the product of the weights of all fixed external edges and mandatory edges. For a choice of internal completions $\boldsymbol\eta=(\eta_e)_{e\in[n]^2}$, denote the corresponding cycle cover by $C(\boldsymbol\eta)$. Contracting directed paths does not change the number of cycles; deleting each internal directed cycle reduces the cycle count by exactly $1$. Hence
\begin{eqnarray*}
  c\bigl(C(\boldsymbol\eta)\bigr)
  &=&c_0+\sum_{e\in[n]^2}\ell_e(\eta_e),\\
  \operatorname{wt}\bigl(C(\boldsymbol\eta)\bigr)
  &=&w_0\prod_{e\in[n]^2}w_e(\eta_e).
\end{eqnarray*}
Thus, after fixing the external edges, summing over all internal completions gives
\begin{eqnarray*}
 &&\sum_{\boldsymbol\eta}
 (-2)^{c(C(\boldsymbol\eta))}\operatorname{wt}(C(\boldsymbol\eta))\\
 &=&(-2)^{c_0}w_0
 \prod_{e\in[n]^2}\left(\sum_{\eta_e}
 (-2)^{\ell_e(\eta_e)}w_e(\eta_e)\right).
\end{eqnarray*}
Therefore the local sum for each six–vertex subgraph is indeed an independent factor. Even if the directed paths belong to a global directed cycle spanning several subgraphs, the factor contributed by the remaining cycles is independent of the internal completion of that subgraph.
Fix a six–vertex subgraph $e=(i,j)$, and abbreviate $I=I_e$, $O=O_e$. Every possible internal completion consists of edges $v_{e,r}\to u_{e,s}$ with weights $W_{rs}$; the $-2$ factor in the local sum is counted only according to the number $\ell_e(\eta_e)$ of directed cycles entirely inside this component.

Suppose that $I=O=\varnothing$. Then each $v_{e,r}$ must point inside the six–vertex subgraph to some $u_{e,s}$, and each of the three $u$ vertices receives exactly one incoming edge. The three mandatory edges $u_{e,r}\to v_{e,r}$ are always present and have weight $1$, so in counting we may regard each pair $(u_{e,r},v_{e,r})$ as a single object labelled $r$. Then choosing edges $v_{e,r}\to u_{e,s}$ corresponds bijectively to permutations in $S_3$, and the number of cycles of the permutation equals the number of directed cycles inside this six–vertex subgraph. Hence the total contribution of all internal edge choices is $\operatorname{Per}_{-2,3}(W)=1$. Suppose $I=O=\{1\}$. Then the external incoming and outgoing edges of label $1$ are already chosen, and we only need to determine the internal edges for labels $2,3$. The sum of all such choices contains the factor
\[
  \operatorname{Per}_{-2}
  \begin{pmatrix}
    W_{22}&W_{23}\\
    W_{32}&W_{33}
  \end{pmatrix}
  =\operatorname{Per}_{-2}
  \begin{pmatrix}
    \frac12&1\\[-1mm]
    -\frac12&-\frac12
  \end{pmatrix}=0.
\]
Thus, regardless of how the external edges outside this subgraph are chosen, the contributions from internal completions with $I=O=\{1\}$ cancel. For $I=O=\{2\}$, the external edges of label $2$ are fixed, and the remaining labels are $1,3$; the relevant submatrix is
\[
  \begin{pmatrix}
    W_{11}&W_{13}\\
    W_{31}&W_{33}
  \end{pmatrix}
  =
  \begin{pmatrix}
    -1&1\\
    1&-\frac12
  \end{pmatrix},
\]
whose $(-2)$-permanent is
\[
  4W_{11}W_{33}-2W_{13}W_{31}
  =4(-1)\left(-\frac12\right)-2(1)(1)=0.
\]
Thus when $I=O$ and both sets are singletons, the local contribution is also $0$.

Suppose that $I=\{1\},O=\{2\}$. Then we must choose edges from $v_{e,1},v_{e,3}$ to $u_{e,2},u_{e,3}$, respectively, and each $u$ vertex is targeted exactly once. Hence there are only two possible completions. The first chooses weights $W_{12}$ and $W_{33}$, forming two directed cycles:
\[
  L_i\to u_{e,1}\to v_{e,1}\to u_{e,2}
  \to v_{e,2}\to L_i,
  \qquad
  u_{e,3}\to v_{e,3}\to u_{e,3}.
\]
The second chooses weights $W_{13}$ and $W_{32}$, forming one directed cycle:
\[
  L_i\to u_{e,1}\to v_{e,1}\to u_{e,3}
  \to v_{e,3}\to u_{e,2}\to v_{e,2}\to L_i.
\]
Both cycle structures use the two external edges $L_i\to u_{e,1}$ and $v_{e,2}\to L_i$, so they are independent of edges outside the subgraph. The two completions differ by one in the number of cycles, so the $(-2)^{c(\pi)}$ factors differ by a factor $-2$. Factoring out the common external edge weight and one $-2$ gives the remaining sum
$(-2)W_{12}W_{33}+W_{13}W_{32}=0$.
Similarly, for $I=\{2\},O=\{1\}$, we choose from $v_{e,2},v_{e,3}$ to $u_{e,1},u_{e,3}$. Choosing $W_{21}$ and $W_{33}$ gives two cycles:
\[
  R_j\to u_{e,2}\to v_{e,2}\to u_{e,1}
  \to v_{e,1}\to R_j,
  \qquad
  u_{e,3}\to v_{e,3}\to u_{e,3}.
\]
Choosing $W_{23}$ and $W_{31}$ gives one cycle:
\[
  R_j\to u_{e,2}\to v_{e,2}\to u_{e,3}
  \to v_{e,3}\to u_{e,1}\to v_{e,1}\to R_j.
\]
Again the number of cycles differs by one, and the remaining sum is
$(-2)W_{21}W_{33}+W_{23}W_{31}=0$.
Hence whenever $I$ and $O$ are both singletons with different labels, the local contribution is also $0$.

Suppose that $I=O=\{1,2\}$. Then all four external edges are chosen. The mandatory edges of labels $1,2$ together with the vertices $L_i,R_j$ form the directed cycle
\[
  L_i\to u_{e,1}\to v_{e,1}\to R_j
  \to u_{e,2}\to v_{e,2}\to L_i .
\]
This directed cycle contributes a factor $-2$, and the edge $L_i\to u_{e,1}$ also contributes the variable factor $a_{ij}$. For labels $3$, the only possible choice is $u_{e,3}\to v_{e,3}\to u_{e,3}$, which forms another directed cycle with contribution $(-2)W_{33}=(-2)\left(-\frac12\right)=1$. Thus the total contribution for $I=O=\{1,2\}$ is $(-2)a_{ij}\cdot 1=-2a_{ij}$.
In summary, denote by
\[
  Z_e(I_e,O_e):=\sum_{\eta_e}
  (-2)^{\ell_e(\eta_e)}w_e(\eta_e)
\]
the local factor corresponding to component $e$ in the factorisation above. The computation shows
\begin{eqnarray*}
  Z_e(I_e,O_e)
  &=&\begin{cases}
    1,&(I_e,O_e)=(\varnothing,\varnothing),\\
    1,&(I_e,O_e)=(\{1,2\},\{1,2\}),\\
    0,&\text{otherwise}.
  \end{cases}
\end{eqnarray*}
In the case $I_e=O_e=\{1,2\}$, the local factor is $(-2)W_{33}=1$; the two contracted boundary paths form a directed cycle through $L_i$ and $R_j$, contributing $-2$, and the external edge $L_i\to u_{e,1}$ contributes $a_{ij}$. Together these give the $-2a_{ij}$ computed above. From the local–global factorisation, if for some six–vertex subgraph the pair $(I_e,O_e)$ is not one of the two nonzero cases, then the total contribution of all cycle covers for that fixed external–edge choice is $0$. Hence we only need to consider subgraphs that either use no external edges or use all four external edges.
For each $i$, among the subgraphs $(i,1),\ldots,(i,n)$ exactly one must use all four external edges. Similarly, for each $j$, among $(1,j),\ldots,(n,j)$ exactly one must use all four. Thus these subgraphs uniquely determine a permutation $\sigma\in S_n$, namely they are exactly $\{(i,\sigma(i)):i\in[n]\}$. Fix such a $\sigma$. After contracting the boundary paths inside each chosen subgraph, for each $i\in[n]$ the component $(i,\sigma(i))$ gives a directed cycle passing through $L_i$ and $R_{\sigma(i)}$. These cycles are vertex–disjoint, and there are no other directed cycles in the contracted graph, so $c_0=n$. The product of the fixed external edge weights is
\[
  w_0=\prod_{i=1}^n a_{i,\sigma(i)},
\]
and all local factors are $1$. Therefore the total contribution of all cycle covers corresponding to $\sigma$ is
\[
  (-2)^n\prod_{i=1}^n a_{i,\sigma(i)}.
\]
Summing over all $\sigma\in S_n$ yields
\begin{eqnarray*}
  \operatorname{Per}_{-2,m(n)}(M_n(A))
  &=&(-2)^n\sum_{\sigma\in S_n}\prod_{i=1}^n a_{i,\sigma(i)}\\
  &=&(-2)^n\operatorname{per}(A).
\end{eqnarray*}
Every nonzero cycle cover contains all mandatory edges $u_{e,r}\to v_{e,r}$. To make the construction completely explicit, for each $i\in[n]$ we fix the mandatory edge $u_{(i,i),1}\longrightarrow v_{(i,i),1}$. These $n$ edges are distinct. Change their weights from $1$ to $-1/2$. This does not change the allowed cycle covers nor the number of cycles in any cover; it merely multiplies the total weight of every nonzero term by $\left(-\frac12\right)^n$.
Denote the modified adjacency matrix by $\widehat M_n(A)$. By linearity of the sum,
\[
  \operatorname{Per}_{-2,m(n)}(\widehat M_n(A))
  =
  (-1/2)^n(-2)^n\operatorname{per}(A)
  =
  \operatorname{per}(A).
\]
Finally, the size of the resulting matrix is $m(n)=6n^2+2n$. Each entry is either one of the original variables $a_{ij}$ or a constant from $\mathbb{F}$. The constants that appear are only
$0,\ 1,\ -1,\ \frac12,\ -\frac12$. Therefore this matrix construction fits the definition of a projection. Since $m(n)$ is quadratic in $n$, these substitutions form a $p$-projection, i.e.,
$(\operatorname{per}_n)_{n\geq1}\leq_p\bigl(\operatorname{Per}_{-2,m}\bigr)_{m\geq1}$.
\end{proof}

\begin{lemma}\label{lem2.7}
Let $X$ and $Y$ be square matrices. Then each of the following holds.
\begin{enumerate}
\renewcommand{\labelenumi}{(\roman{enumi})}
  \item If $X$ is $m\times m$, then
  \begin{eqnarray*}
    \Phi_X(z)=z^m\Phi_X(z^{-1}).
  \end{eqnarray*}
  \item The polynomial $\Phi_X(z)$ is multiplicative under direct sums:
  \begin{eqnarray*}
    \Phi_{X\oplus Y}(z)=\Phi_X(z)\Phi_Y(z).
  \end{eqnarray*}
  \item If $X$ is $m\times m$, then
  \begin{eqnarray*}
    \Phi_X(1)=(-1)^m\operatorname{Per}_{-2,m}(X).
  \end{eqnarray*}
\end{enumerate}
\end{lemma}

\begin{proof}
$(i)$ First compare $D_k(X)$ and $D_{m-k}(X)$. The map $S\mapsto S^{\mathrm{c}}$ gives a bijection between the $k$-element subsets of $[m]$ and the $(m-k)$-element subsets, and the product $\operatorname{det} X[S,S]\operatorname{det} X[S^{\mathrm{c}},S^{\mathrm{c}}]$ is unchanged when $S$ and $S^{\mathrm{c}}$ are interchanged. Hence
\[
  D_k(X)=D_{m-k}(X),\qquad 0\leq k\leq m.
\]
Thus
\begin{eqnarray*}
  z^m\Phi_X(z^{-1})
  &=&z^m\sum_{k=0}^mD_k(X)z^{-k}
    =\sum_{k=0}^mD_k(X)z^{m-k}\\
  &=&\sum_{j=0}^mD_{m-j}(X)z^j
    =\sum_{j=0}^mD_j(X)z^j
    =\Phi_X(z),
\end{eqnarray*}
where the third equality uses the substitution $j=m-k$, and the fourth uses $D_{m-j}=D_j$. This shows that the coefficient sequence
$D_0(X),D_1(X),\ldots,D_m(X)$ is symmetric with respect to the formal index $m$, i.e.,
$\Phi_X(z)=z^m\Phi_X(z^{-1})$. This does not claim that the actual degree of $\Phi_X$ is $m$; for example, if $\det X=0$, its leading coefficient may vanish. We shall only use the above self–inverse identity and the coefficient relation $D_j(X)=D_{m-j}(X)$.

$(ii)$ Suppose $X$ and $Y$ have sizes $p$ and $q$, respectively. For any principal submatrix of $X\oplus Y$ corresponding to a subset $T\subseteq[p+q]$, we can uniquely extract a subset $S\subseteq[p]$ from the first block and a subset $R\subseteq[q]$ from the second block:
\[
  S=T\cap[p],
  \qquad
  R=\{j\in[q]:p+j\in T\}.
\]
Here $S^{\mathrm{c}}$ denotes the complement of $S$ in $[p]$, and $R^{\mathrm{c}}$ the complement of $R$ in $[q]$. After subtracting $p$ from the row and column indices in the second block, we have
$(X\oplus Y)[T,T]=X[S,S]\oplus Y[R,R]$. Hence
\[
  \det (X\oplus Y)[T,T]
  =\det X[S,S]\det Y[R,R].
\]
The complement decomposes in the same way, and $|T|=|S|+|R|$. Therefore
\begin{eqnarray*}
  \Phi_{X\oplus Y}(z)
  &=&\sum_{S\subseteq[p]}\sum_{R\subseteq[q]}
    \det X[S,S]\det X[S^{\mathrm{c}},S^{\mathrm{c}}]\\
  &&{}\cdot
    \det Y[R,R]\det Y[R^{\mathrm{c}},R^{\mathrm{c}}]
    z^{|S|+|R|}\\
  &=&\left(\sum_{S\subseteq[p]}
    \det X[S,S]\det X[S^{\mathrm{c}},S^{\mathrm{c}}]z^{|S|}\right)\\
  &&{}\cdot
    \left(\sum_{R\subseteq[q]}
    \det Y[R,R]\det Y[R^{\mathrm{c}},R^{\mathrm{c}}]z^{|R|}\right)\\
  &=&\Phi_X(z)\Phi_Y(z).
\end{eqnarray*}

$(iii)$ Let $X=[x_{ij}]$ be $m\times m$. Then
\[
  \Phi_X(1)=\sum_{S\subseteq[m]}
  \det X[S,S]\det X[S^{\mathrm{c}},S^{\mathrm{c}}].
\]
Fix $S$ and expand the two determinants using Leibniz' formula. The first determinant chooses a permutation $\sigma$ on $S$, the second chooses a permutation $\tau$ on $S^{\mathrm{c}}$. Combining $\sigma$ and $\tau$ gives a permutation $\pi\in S_m$ that stabilises $S$. Conversely, if $\pi(S)=S$, then the restrictions of $\pi$ to $S$ and $S^{\mathrm{c}}$ uniquely determine $\sigma$ and $\tau$. Since
\[
  \operatorname{sgn}(\sigma)\operatorname{sgn}(\tau)=\operatorname{sgn}(\pi),
  \qquad
  \left(\prod_{i\in S}x_{i\sigma(i)}\right)
  \left(\prod_{i\in S^{\mathrm{c}}}x_{i\tau(i)}\right)
  =\prod_{i=1}^m x_{i\pi(i)},
\]
we can first fix $\pi$ and count the number of subsets $S$ with $\pi(S)=S$.
Write $\pi$ as a product of disjoint cycles. A subset $S$ is invariant under $\pi$ iff for each cycle of $\pi$, $S$ contains either all elements of that cycle or none. If $\pi$ has $c(\pi)$ cycles, then each cycle can be independently chosen to be included in $S$ or not. Hence the number of such $S$ is $2^{c(\pi)}$. For all these $S$, the sign and the permutation monomial are the same. Therefore
\begin{eqnarray*}
  \Phi_X(1)
  &=&\sum_{\pi\in S_m}\operatorname{sgn}(\pi)2^{c(\pi)}
    \prod_i x_{i\pi(i)}  \\
  &=&(-1)^m\sum_{\pi\in S_m}(-2)^{c(\pi)}
    \prod_i x_{i\pi(i)}.
\end{eqnarray*}
The second equality comes from the following sign computation. The sign of a cycle of length $\ell$ is $(-1)^{\ell-1}$. Multiplying the signs of all disjoint cycles of $\pi$ gives $\operatorname{sgn}(\pi)=(-1)^{m-c(\pi)}$.
Thus
\[
  \operatorname{sgn}(\pi)2^{c(\pi)}
  =(-1)^{m-c(\pi)}2^{c(\pi)}
  =(-1)^m(-2)^{c(\pi)}.
\]
By the definition of the $(-2)$-permanent, we obtain $\Phi_X(1)=(-1)^m\operatorname{Per}_{-2,m}(X)$.
\end{proof}

\begin{lemma}\label{lem2.8}
Let $m\geq1$, $q=q(m)$, $L=L(m)$, $r=r(m)$ and $N=N(m)$. Then $N(m)$ is an integer, and for every polynomial $H\in\mathbb{F}[z]$ satisfying
\begin{eqnarray}\label{equ10}
  H(z)=z^mH(z^{-1})
\end{eqnarray}
we have
\begin{eqnarray}\label{equ11}
  \bigl([z^N]-[z^{N-1}]\bigr)H(z)K_m(z)=H(1).
\end{eqnarray}
\end{lemma}
\begin{proof}
We first examine the degree, parity and coefficient symmetry of $K_m$. Since $L=2^q-1$ is odd, when $m$ is odd,
$K_m(z)=S_L(z)^3$ has degree $r=3L$, still odd. When $m$ is even, $K_m(z)=\frac{1+z}{2}S_L(z)^3$ has degree $r=3L+1$, even.
Thus in both cases $m$ and $r$ have the same parity, so $m+r$ is even and $N=(m+r)/2$ is indeed an integer. Since $z^LS_L(z^{-1})=S_L(z)$, $S_L$ is palindromic of degree $L$. Also $(1+z)/2$ is palindromic of degree $1$. Since products of palindromic polynomials are palindromic, in both cases $K_m$ is palindromic of degree $r$.
Write
\[
  S_L(z)^3=\sum_j b_jz^j.
\]
Using
\[
  S_L(z)^3=\frac{(1-z^{L+1})^3}{(1-z)^3},
  \qquad
  \frac1{(1-z)^3}=\sum_{t\geq0}\binom{t+2}{2}z^t,
\]
where $\binom{u}{2}=u(u-1)/2$. The second equality is the binomial expansion in the sense of formal power series, with coefficient $\binom{t+2}{2}$ for $z^t$.
Expanding the numerator gives $(1-z^{L+1})^3=1-3z^{L+1}+3z^{2L+2}-z^{3L+3}$.
For $L+1\leq j\leq2L+1$, the last two terms have degree at least $2L+2$, so they do not affect the coefficient of $z^j$. Hence
\begin{eqnarray}\label{equ12}
  b_j=\binom{j+2}{2}-3\binom{j-L+1}{2}.
\end{eqnarray}
For any sequence $f_j$, define its central second difference as
\[
  \Delta^2f_j=f_{j+1}-2f_j+f_{j-1}.
\]
A direct computation gives
\[
  \binom{j+3}{2}-2\binom{j+2}{2}+\binom{j+1}{2}=1.
\]
Thus the central second difference of the quadratic polynomial $\binom{j+2}{2}$ is $1$. Translating the variable by $L+1$ does not change the second difference, so $\binom{j-L+1}{2}$ also has central second difference $1$. From \eqref{equ12}, the central second difference of $b_j$ is therefore $1-3=-2$.
To ensure that $j-1,j,j+1$ all lie in the range where \eqref{equ12} applies, take $L+2\leq j\leq2L$, giving
\begin{eqnarray}\label{equ13}
  \Delta^2 b_j=-2,
  \qquad L+2\leq j\leq 2L.
\end{eqnarray}
Now write $K_m(z)=\sum_jk_jz^j$. When $m$ is odd, $k_j=b_j$. When $m$ is even, $k_j=\frac{b_j+b_{j-1}}2$. Since the second difference is linear,
\[
  \Delta^2k_j
  =\frac12\bigl(\Delta^2b_j+\Delta^2b_{j-1}\bigr).
\]
For $L+3\leq j\leq2L$, both $j$ and $j-1$ lie in the range of \eqref{equ13}, hence
\begin{eqnarray}\label{equ14}
  \Delta^2k_j=-2,
  \qquad L+3\leq j\leq 2L.
\end{eqnarray}
Interpreting $H(z)=z^mH(z^{-1})$ as an identity in Laurent polynomials, since the left side contains no negative powers, $H$ has zero coefficients outside the degree range $0,1,\ldots,m$. Within this range, it also says that the coefficients of $z^j$ and $z^{m-j}$ in $H$ are equal. First consider the case $m=2d$. Then $r=3L+1$ is even, write $r=2C$, so $N=d+C$. Here
\[
  C=\frac{3L+1}{2},\qquad d=\frac m2.
\]
Since $L\geq2m+7$, we have $L\geq m+5$, hence
\begin{eqnarray*}
  C-d&=&\frac{3L+1-m}{2}\geq L+3,\\
  C&\leq&2L.
\end{eqnarray*}
Thus for every $0\leq u\leq d$, we have $L+3\leq C-u\leq2L$. Therefore, when applying \eqref{equ14} at $j=C$ and at $j=C-u$ below, the index range is satisfied. Because the coefficients of $H$ are symmetric about the midpoint $d$, we may pair equal coefficients. Hence $H$ can be uniquely written as
\begin{eqnarray}\label{equ15}
  H(z)=h_0z^d+
  \sum_{u=1}^{d}h_u\bigl(z^{d-u}+z^{d+u}\bigr).
\end{eqnarray}
Since $K_m$ is palindromic of degree $2C$, for every integer $t$ with $|t|\leq C$ we have $k_{C+t}=k_{C-t}$. In particular, $k_{C+1}=k_{C-1}$. Setting $j=C$ in \eqref{equ14} gives $k_{C+1}-2k_C+k_{C-1}=-2$.
Using $k_{C+1}=k_{C-1}$, we obtain $k_C-k_{C-1}=1$. Therefore, for the middle monomial $z^d$,
\begin{eqnarray*}
  \bigl([z^N]-[z^{N-1}]\bigr)z^dK_m=k_C-k_{C-1}=1.
\end{eqnarray*}
Now fix $1\leq u\leq d$. Extracting coefficients by shifting the exponent gives
\begin{eqnarray*}
  &&\bigl([z^N]-[z^{N-1}]\bigr)
    \bigl(z^{d-u}+z^{d+u}\bigr)K_m\notag\\
  &=&(k_{C+u}-k_{C+u-1})
       +(k_{C-u}-k_{C-u-1})\notag\\
  &=&2k_{C-u}-k_{C-u+1}-k_{C-u-1}=2.
\end{eqnarray*}
The second equality uses the coefficient symmetry of degree $2C$:
\[
  k_{C+u}=k_{C-u},
  \qquad
  k_{C+u-1}=k_{C-u+1}.
\]
The last equality uses $2k_j-k_{j+1}-k_{j-1}=-\Delta^2k_j=2$, where $j=C-u$.
Applying the above computation to each term in the decomposition \eqref{equ15} yields
\[
  \bigl([z^N]-[z^{N-1}]\bigr)H(z)K_m(z)
  =h_0+2\sum_{u=1}^{d}h_u.
\]
On the other hand, setting $z=1$ in \eqref{equ15} gives exactly
$H(1)=h_0+2\sum_{u=1}^dh_u$. Hence when $m$ is even, \eqref{equ11} holds.
Now consider $m=2d+1$. Then $r=3L$ is odd, write $r=2C+1$, and $N=d+C+1$. Here
\[
  C=\frac{3L-1}{2},\qquad d=\frac{m-1}{2}.
\]
Since $L\geq2m+7$, we have $L\geq m+4$, hence
\begin{eqnarray*}
  C-d&=&\frac{3L-m}{2}\geq L+2,\\
  C&\leq&2L.
\end{eqnarray*}
Thus for every $0\leq u\leq d$, we have $L+2\leq C-u\leq2L$. Therefore \eqref{equ13} applies at $j=C-u$. Because the coefficients of $z^j$ and $z^{m-j}$ in $H$ are equal, $H$ can be uniquely written as
\begin{eqnarray*}
  H(z)=\sum_{u=0}^{d}h_u
  \bigl(z^{d-u}+z^{d+1+u}\bigr).
\end{eqnarray*}
For each pair of monomials, extracting the coefficients of $z^N$ and $z^{N-1}$ gives
\begin{eqnarray*}
  &&\bigl([z^N]-[z^{N-1}]\bigr)
  \bigl(z^{d-u}+z^{d+1+u}\bigr)K_m\notag\\
  &=&k_{C+1+u}-k_{C+u}
       +k_{C-u}-k_{C-u-1}\notag\\
  &=&2k_{C-u}-k_{C-u+1}-k_{C-u-1}=2.
\end{eqnarray*}
Since $K_m=S_L^3$ has degree $2C+1$, its coefficients are symmetric about the half–integer centre $C+1/2$, i.e.,
$k_{C+1+u}=k_{C-u}$ and $k_{C+u}=k_{C-u+1}$.
Also, when $m$ is odd, $k_j=b_j$. Applying \eqref{equ13} at $j=C-u$ gives
$-\Delta^2k_{C-u}=2$. Summing over $u=0,1,\ldots,d$ yields
\[
  \bigl([z^N]-[z^{N-1}]\bigr)H(z)K_m(z)
  =2\sum_{u=0}^{d}h_u=H(1).
\]
Thus when $m$ is odd, \eqref{equ11} also holds.

In summary, \eqref{equ11} holds for both parities of $m$. This completes the proof. 
\end{proof}

\begin{lemma}\label{lem2.9}
Let $q$ be a positive integer, $L=2^q-1$, and let $P_L^{[\varepsilon]}$ be the monomial matrix. Then for each $\varepsilon\in\{1,-1\}$,
\begin{eqnarray*}
  \Phi_{P_L^{[\varepsilon]}}(z)
  =\varepsilon(1+z+\cdots+z^L).
\end{eqnarray*}
\end{lemma}

\begin{proof}
Write $P=P_L^{[\varepsilon]}$. By definition of $P_L^{[\varepsilon]}$, the associated permutation of $P$ is $\rho_L$, whose disjoint cycles have lengths
$1,2,4,\ldots,2^{q-1}$.
For any $S\subseteq[L]$, the principal submatrix $P[S,S]$ has nonzero determinant iff for every $i\in S$, the column $\rho_L(i)$ containing the unique nonzero entry of row $i$ also belongs to $S$, i.e., $\rho_L(S)=S$. From the cycle decomposition of $\rho_L$, this condition is equivalent to $S$ being a union of entire cycles of $\rho_L$. Indeed, if $\rho_L(S)\ne S$, then there exists $i\in S$ with
$\rho_L(i)\notin S$, so row $i$ of $P[S,S]$ is zero, making the determinant zero. If $\rho_L(S)=S$, then $P[S,S]$ is again a monomial matrix, hence its determinant is nonzero.
When $S$ satisfies this condition, both $P[S,S]$ and $P[S^{\mathrm{c}},S^{\mathrm{c}}]$ are monomial matrices, so each Leibniz expansion has only one nonzero term. Therefore
\begin{eqnarray*}
  \det P[S,S]
  &=&\operatorname{sgn}(\rho_L|_S)\prod_{i\in S}\omega_i,\\
  \det P[S^{\mathrm{c}},S^{\mathrm{c}}]
  &=&\operatorname{sgn}(\rho_L|_{S^{\mathrm{c}}})
    \prod_{i\in S^{\mathrm{c}}}\omega_i.
\end{eqnarray*}
These two restricted permutations act on disjoint sets; their product is $\rho_L$, hence
\[
  \det P[S,S]\det P[S^{\mathrm{c}},S^{\mathrm{c}}]
  =\operatorname{sgn}(\rho_L)\prod_i\omega_i.
\]
If the indices of $S$ and $S^{\mathrm{c}}$ are interleaved in the natural order, we may apply the same permutation to the rows and columns of $P$ to move the rows and columns corresponding to $S$ to the front and those of $S^{\mathrm{c}}$ to the back. Simultaneous row and column permutations do not change the determinant. After reordering, the associated permutation of the monomial matrix is the block direct sum of $\rho_L|_S$ and $\rho_L|_{S^{\mathrm{c}}}$, so its sign is the product of the signs of the two restrictions. Denote this constant (independent of $S$) by
\[
  \gamma(P)=\operatorname{sgn}(\rho_L)\prod_i\omega_i.
\]
From $\rho_L(S)=S$, for each cycle of $\rho_L$ of length $2^j$, the set $S$ has exactly two choices: either include all elements of that cycle or none. The first choice increases $|S|$ by $2^j$, hence contributes a factor $z^{2^j}$ in the generating polynomial; the second contributes $1$. The cycles are independent, so
\[
  \Phi_P(z)=\gamma(P)\prod_{j=0}^{q-1}(1+z^{2^j}).
\]
Using the identity
\[
  (1-z)\prod_{j=0}^{q-1}(1+z^{2^j})=1-z^{2^q},
\]
or equivalently the binary expansion identity, we get
\[
  \prod_{j=0}^{q-1}(1+z^{2^j})=1+z+\cdots+z^L.
\]
By definition of $P_L^{[\varepsilon]}$,
\[
  \prod_i\omega_i=\varepsilon\,\operatorname{sgn}(\rho_L),
\]
so $\gamma(P)=\varepsilon$. Hence $\Phi_{P_L^{[\varepsilon]}}(z)=\varepsilon(1+z+\cdots+z^L)$.
\end{proof}

In general, for a $1\times1$ matrix $[c]$, we have $\Phi_{[c]}(z)=c(1+z)$, so $\Phi_{[1/2]}(z)=(1+z)/2$. Combining Lemma \ref{lem2.7}(ii) and Lemma \ref{lem2.9}, when $m$ is odd, $\Phi_{\mathcal B_m}(z)=\bigl(-S_L(z)\bigr)S_L(z)S_L(z)=-S_L(z)^3=-K_m(z)$,
and when $m$ is even,
$\Phi_{\mathcal B_m}(z)=S_L(z)^3\frac{1+z}{2}=K_m(z)$.
Thus both cases can be written uniformly as
\begin{eqnarray}\label{equ16}
  \Phi_{\mathcal B_m}(z)=(-1)^mK_m(z).
\end{eqnarray}

\begin{lemma}(Sagan, \cite{sag})\label{lem2.02}
Let $r$ and $s$ be two positive  integers and $\lambda$ a partition of $r+s$. Then
\begin{eqnarray*}
\operatorname{Res}^{S_{r+s}}_{S_r \times S_s} \chi_{\lambda}
= \sum_{\substack{\mu \vdash r \\ \nu \vdash s}}
  c_{\mu,\nu}^{\lambda} \,
  \chi_{\mu} \otimes \chi_{\nu},
\end{eqnarray*}
where $\operatorname{Res}^{S_{r+s}}_{S_r \times S_s}\chi_\lambda$ denotes the restriction of the character $\chi_\lambda$ to the Young subgroup $S_r\times S_s$, $\chi_{\mu}\otimes\chi_{\nu}$ is the outer tensor product character defined earlier, and $c_{\mu,\nu}^{\lambda}$ are the Littlewood–Richardson coefficients.
\end{lemma}

\begin{lemma}\label{lem2.4}
Let $r$ and $s$ be two positive  integers, $\lambda$ a partition of $r+s$. Let
$A=[a_{ij}]\in\mathbb{F}^{r\times r}$, $B=[b_{ij}]\in\mathbb{F}^{s\times s}$.
Then
\begin{eqnarray*}
  \operatorname{Imm}_\lambda(A\oplus B)
  =\sum_{\substack{\mu\vdash r\\\nu\vdash s}}
    c_{\mu,\nu}^{\lambda}\,
    \operatorname{Imm}_\mu(A)\operatorname{Imm}_\nu(B).
\end{eqnarray*}
\end{lemma}

\begin{proof}
Partition the row and column index sets of $A\oplus B$ into the two disjoint parts
\[
  I_A=\{1,\ldots,r\},
  \qquad
  I_B=\{r+1,\ldots,r+s\}.
\]
Note that the two off–diagonal blocks of $A\oplus B$ are zero. If for some permutation $\pi\in S_{r+s}$ there exists $i\in I_A$ with $\pi(i)\in I_B$, then the permutation monomial $\prod_i(A\oplus B)_{i,\pi(i)}$ contains a zero factor. Hence only permutations that stabilise both $I_A$ and $I_B$ can contribute nonzero terms. Each such permutation can be uniquely written as $\sigma\times\tau$, where $\sigma\in S_r$ acts on the first block and $\tau\in S_s$ on the second. More precisely, $\sigma\times\tau\in S_{r+s}$ is defined by
\[
  (\sigma\times\tau)(i)=
  \begin{cases}
    \sigma(i),&1\leq i\leq r,\\
    r+\tau(i-r),&r+1\leq i\leq r+s.
  \end{cases}
\]
Thus
\[
  \operatorname{Imm}_\lambda(A\oplus B)
  =\sum_{\sigma\in S_r}\sum_{\tau\in S_s}
  \chi_{\lambda}(\sigma\times\tau)
  \prod_{i=1}^r a_{i\sigma(i)}
  \prod_{j=1}^s b_{j\tau(j)}.
\]
By Lemma \ref{lem2.02}, the restriction of the irreducible character $\chi_\lambda$ to the Young subgroup $S_r\times S_s$ is
\[
  \operatorname{Res}^{S_{r+s}}_{S_r\times S_s}\chi_\lambda
  =\sum_{\substack{\mu\vdash r\\\nu\vdash s}}
    c_{\mu,\nu}^{\lambda}\,\chi_\mu\otimes\chi_\nu.
\]
By definition of the outer tensor product character, $(\chi_\mu\otimes\chi_\nu)(\sigma,\tau)=\chi_\mu(\sigma)\chi_\nu(\tau)$. Hence for every $\sigma\in S_r$ and $\tau\in S_s$,
\[
  \chi_{\lambda}(\sigma\times\tau)
  =\sum_{\substack{\mu\vdash r\\\nu\vdash s}}
    c_{\mu,\nu}^{\lambda}\,\chi_{\mu}(\sigma)\chi_{\nu}(\tau).
\]
Since all sums are finite, we may sum over $\mu,\nu$ first and then over $\sigma,\tau$:
\begin{eqnarray*}
  \operatorname{Imm}_\lambda(A\oplus B)
  &=&\sum_{\substack{\mu\vdash r\\\nu\vdash s}}
    c_{\mu,\nu}^{\lambda}
    \left(\sum_{\sigma\in S_r}\chi_{\mu}(\sigma)
      \prod_{i=1}^r a_{i\sigma(i)}\right)
    \left(\sum_{\tau\in S_s}\chi_{\nu}(\tau)
      \prod_{j=1}^s b_{j\tau(j)}\right)  \\
  &=&\sum_{\substack{\mu\vdash r\\\nu\vdash s}}
    c_{\mu,\nu}^{\lambda}\,\operatorname{Imm}_\mu(A)\operatorname{Imm}_\nu(B).
\end{eqnarray*}
This completes the proof. 
\end{proof}

\section{Proofs of Theorems \ref{thm1} and \ref{thm2}}

Before proving Theorems \ref{thm1} and \ref{thm2}, we first give two important lemmas.

\begin{lemma}\label{lem2.10}
There exist a function $N=N(m)$ and a family of constant matrices $(\mathcal B_m)_{m\geq1}$, where $N(m)=O(m)$, such that for every positive integer $m$ and every $m\times m$ generic variable matrix $X=(x_{ij})$,
\begin{eqnarray*}
  \operatorname{Imm}_{(2^{N(m)})}(X\oplus\mathcal B_m)
  =\operatorname{Per}_{-2,m}(X).
\end{eqnarray*}
Moreover, $\bigl(\operatorname{Per}_{-2,m}\bigr)_{m\geq1}\leq_p
\bigl(\operatorname{Imm}_{(2^n)}\bigr)_{n\geq1}$.
\end{lemma}

\begin{proof}
Fix $m$, and write $q=q(m)$, $L=L(m)$, $r=r(m)$ and $N=N(m)$. Take the constant block $\mathcal B_m$ defined above. Since $\dim\mathcal B_m=r$ and $N=(m+r)/2$, the direct sum matrix $X\oplus\mathcal B_m$ has size $m+r=2N$.
Hence its size equals the size of the two–column rectangular partition $(2^N)$, so we may apply Lemma \ref{lem2.2}(ii) to obtain
$\operatorname{Imm}_{(2^N)}(X\oplus\mathcal B_m)=D_N(X\oplus\mathcal B_m)-D_{N-1}(X\oplus\mathcal B_m)$.
By the definition of the generating polynomial
$\Phi_Z(z)=\sum_kD_k(Z)z^k$ (for any square matrix $Z$), $D_N(Z)$ and $D_{N-1}(Z)$ are respectively the coefficients of $z^N$ and $z^{N-1}$. Thus the right–hand side above can be written as $\bigl([z^N]-[z^{N-1}]\bigr)\Phi_{X\oplus\mathcal B_m}(z)$.
By Lemma \ref{lem2.7}(ii), $\Phi$ is multiplicative under direct sums. Hence we may separate the blocks $X$ and $\mathcal B_m$; combining with equation \eqref{equ16} gives
$\Phi_{X\oplus\mathcal B_m}(z)=\Phi_X(z)\Phi_{\mathcal B_m}(z)=(-1)^m\Phi_X(z)K_m(z)$.
By Lemma \ref{lem2.7}(i), we have $\Phi_X(z)=z^m\Phi_X(z^{-1})$, which means $H(z)=\Phi_X(z)$ satisfies the coefficient symmetry condition required in Lemma \ref{lem2.8}. Combining the above identities yields
\begin{eqnarray*}
  \operatorname{Imm}_{(2^N)}(X\oplus\mathcal B_m)
  &=&(-1)^m\bigl([z^N]-[z^{N-1}]\bigr)
    \Phi_X(z)K_m(z) \\
  &=&(-1)^m\Phi_X(1)=\operatorname{Per}_{-2,m}(X).
\end{eqnarray*}
Here the second equality uses Lemma \ref{lem2.8} on the two middle coefficients:
$$\bigl([z^N]-[z^{N-1}]\bigr)\Phi_X(z)K_m(z)=\Phi_X(1),$$
and the third uses Lemma \ref{lem2.7}(iii): $\Phi_X(1)=(-1)^m\operatorname{Per}_{-2,m}(X)$. In the matrix $X\oplus\mathcal B_m$, the upper–left block is exactly $X$, the two off–diagonal blocks are zero, and the lower–right block $\mathcal B_m$ consists of constants in $\mathbb{F}$. Hence, starting from a generic $2N\times 2N$ matrix $Y$, we need only assign the upper–left block to the corresponding variables of $X$, set the two off–diagonal blocks to $0$, and set the lower–right block to the corresponding constants of $\mathcal B_m$, thus obtaining $X\oplus\mathcal B_m$. By the minimality of $q=q(m)$, we have $2^{q-1}-1<2m+7$. A simple calculation gives $2^q<4m+16$, hence $L=2^q-1<4m+15$.
Depending on the parity of $m$, $r(m)$ equals $3L$ or $3L+1$, so $r(m)\leq3L+1=O(m)$. Therefore $N(m)=\frac{m+r(m)}{2}=O(m)$.
Moreover, $q(m)$, $L(m)$ and the positions of each constant in $\mathcal B_m$ are uniquely determined by $m$. Hence $m\mapsto N(m)$ is an integer–valued $p$-bounded index function, and $(\mathcal B_m)_{m\geq1}$ forms a well–defined family of constant matrices. Since $N(m)=O(m)$, the size $2N(m)$ of the resulting matrix grows at most linearly in $m$, so the above projection is a $p$-projection.
\end{proof}

\begin{lemma}\label{lem2.5}
Let $n$ be a positive integer and $A=(a_{ij})$ an $n\times n$ generic variable matrix. Then
\begin{eqnarray*}
  \operatorname{Imm}_{(n,n)}(A\oplus U_n)=\operatorname{per}(A).
\end{eqnarray*}
Moreover, $(\operatorname{per}_n)_{n\geq1}\leq_p\bigl(\operatorname{Imm}_{(n,n)}\bigr)_{n\geq1}$.
\end{lemma}

\begin{proof}
The two diagonal blocks of $A\oplus U_n$ both have size $n$, and the partition $(n,n)$ has size $2n$. Thus, taking $r=s=n$ and $\lambda=(n,n)$ in Lemma \ref{lem2.4} gives
\[
  \operatorname{Imm}_{(n,n)}(A\oplus U_n)
  =\sum_{\substack{\mu\vdash n\\\nu\vdash n}}
    c_{\mu,\nu}^{(n,n)}\operatorname{Imm}_\mu(A)\operatorname{Imm}_\nu(U_n).
\]
By equation \eqref{equ5}, $\operatorname{Imm}_\nu(U_n)$ is $1$ only for $\nu=(n)$ and $0$ otherwise. Hence
\begin{eqnarray}\label{equ6}
  \operatorname{Imm}_{(n,n)}(A\oplus U_n)
  =\sum_{\mu\vdash n}c_{\mu,(n)}^{(n,n)}\operatorname{Imm}_\mu(A).
\end{eqnarray}
Since $(n)$ is a one–row partition, Pieri's rule says that $c_{\mu,(n)}^{(n,n)}=1$ if and only if the skew Young diagram $(n,n)/\mu$ is a horizontal strip of size $n$; otherwise the coefficient is $0$. Here $(n,n)/\mu$ denotes the boxes remaining after removing the Young diagram $\mu$ from the upper–left corner of the two–row rectangle $(n,n)$. The rectangle $(n,n)$ has $n$ columns. A horizontal strip requires at most one new box per column, and $(n,n)/\mu$ has exactly $n$ boxes, so each column must contain exactly one box. Each column of the original rectangle has two boxes; hence, after removing the skew diagram, the remaining $\mu$ must also occupy exactly one box in each column. Since $\mu$ is a left–justified Young diagram with nonincreasing row lengths, these boxes can only form a full first row, i.e., $\mu=(n)$. Conversely, if $\mu=(n)$, then $(n,n)/\mu$ consists exactly of the $n$ boxes in the second row, one per column, which is indeed a horizontal strip. Therefore
\[
  c_{\mu,(n)}^{(n,n)}=\begin{cases}
    1,&\mu=(n),\\
    0,&\mu\ne(n).
  \end{cases}
\]
Substituting this into equation \eqref{equ6}, the sum reduces to the single term $\mu=(n)$:
\[
  \operatorname{Imm}_{(n,n)}(A\oplus U_n)=\operatorname{Imm}_{(n)}(A)=\operatorname{per}(A).
\]
The last equality holds because the partition $(n)$ corresponds to the trivial character of $S_n$, whose value is identically $1$, so the corresponding immanant is precisely the permanent. In a generic $2n\times 2n$ matrix, setting the upper–left block to $A$, the two off–diagonal blocks to zero, and the lower–right block to $U_n$ yields exactly $A\oplus U_n$. Here the upper–left block uses only the original variables $a_{ij}$, while the other three blocks use only constants $0$, $1$ and $1/n!$. Since $\mathbb{F}$ has characteristic zero, $n!$ is invertible in $\mathbb{F}$, so all these constants belong to $\mathbb{F}$. Since the resulting matrix has size $2n$, this construction gives a projection, and it forms a $p$-projection.
\end{proof}

\noindent\textbf{Remark.}
If we replace $U_n$ by the all–one matrix $J_n$, the same argument yields
$\operatorname{Imm}_{(n,n)}(A\oplus J_n)=n!\,\operatorname{per}(A)$.
Taking the first row of $U_n$ to be $1/n!$ exactly removes this scalar factor, giving
$\operatorname{Imm}_{(n,n)}(A\oplus U_n)=\operatorname{per}(A)$.

\textbf{Proof of Theorem \ref{thm1}.}
By Lemma \ref{lem2.3}, we have $\bigl(\operatorname{Imm}_{(2^n)}\bigr)_{n\geq1}\in\mathrm{VNP}$. Hence it remains to prove
$(\operatorname{per}_n)_{n\geq1}\leq_p\bigl(\operatorname{Imm}_{(2^n)}\bigr)_{n\geq1}$.
Let $A=[a_{ij}]$ be an $n\times n$ generic variable matrix. By Lemma \ref{lem2.6}, there exists a matrix $\widehat M_n(A)$ of size $m(n)=6n^2+2n$, each entry of which is either one of the variables $a_{ij}$ or a constant from $\mathbb{F}$, such that $\operatorname{Per}_{-2,m(n)}\bigl(\widehat M_n(A)\bigr)=\operatorname{per}(A)$.
Now take $X=\widehat M_n(A)$ in Lemma \ref{lem2.10}. By Lemma \ref{lem2.10}, there exist a constant matrix $\mathcal B_{m(n)}$ and an integer $N(n)=O(m(n))$ such that
\[
  \operatorname{Imm}_{(2^{N(n)})}
  \bigl(\widehat M_n(A)\oplus\mathcal B_{m(n)}\bigr)
  =\operatorname{Per}_{-2,m(n)}\bigl(\widehat M_n(A)\bigr)
  =\operatorname{per}(A).
\]
Define $\Pi_n(A):=\widehat M_n(A)\oplus\mathcal B_{m(n)}$.
From the size relation $m(n)+\dim\mathcal B_{m(n)}=2N(n)$ in Lemma \ref{lem2.10}, $\Pi_n(A)$ has size $2N(n)$. This agrees with the size of the partition $(2^{N(n)})$, so $\operatorname{Imm}_{(2^{N(n)})}(\Pi_n(A))$ is well–defined. Every entry of $\widehat M_n(A)$ is either one of the variables $a_{ij}$ or a constant from $\mathbb{F}$. The entries of the constant matrix $\mathcal B_{m(n)}$ and the zero off–diagonal blocks are also constants in $\mathbb{F}$. Thus the map $A\mapsto\Pi_n(A)$ is a projection. Since $m(n)=O(n^2)$ and $N(n)=O(m(n))=O(n^2)$, the size of the matrix is polynomially bounded in $n$, so $(\Pi_n)_{n\geq1}$ forms a $p$-projection.
Over a field of characteristic zero, the permanent family is VNP–complete. The above identity proves
$(\operatorname{per}_n)_{n\geq1}\leq_p
  \bigl(\operatorname{Imm}_{(2^n)}\bigr)_{n\geq1}$.
Therefore, for any polynomial family $f\in\mathrm{VNP}$, by transitivity of $p$-projections,
\[
  f\leq_p(\operatorname{per}_n)_{n\geq1}\leq_p
  \bigl(\operatorname{Imm}_{(2^n)}\bigr)_{n\geq1}.
\]
This shows that the immanant family is VNP–hard. Since by Lemma \ref{lem2.3} it belongs to VNP, it is VNP–complete under $p$-projections. Moreover, the above construction gives $N(n)=O(n^2)$ and
$\operatorname{Imm}_{(2^{N(n)})}(\Pi_n(A))=\operatorname{per}(A)$.\qed

\textbf{Proof of Theorem \ref{thm2}.}
By Lemma \ref{lem2.3}, we have $\bigl(\operatorname{Imm}_{(n,n)}\bigr)_{n\geq1}\in\mathrm{VNP}$.
Take the constant matrix $U_n\in\mathbb{F}^{n\times n}$ constructed in Lemma \ref{lem2.5}. For any $n\times n$ generic variable matrix $A$, Lemma \ref{lem2.5} gives $\operatorname{Imm}_{(n,n)}(A\oplus U_n)=\operatorname{per}(A)$.
Starting from a generic $2n\times 2n$ matrix, setting the upper–left block to $A$, the two off–diagonal blocks to zero, and the lower–right block to the constant matrix $U_n$ yields $A\oplus U_n$. This construction uses only the variables in $A$ and constants from $\mathbb{F}$, and the resulting matrix has size $2n$. Hence, 
$ (\operatorname{per}_n)_{n\geq1}\leq_p
  \bigl(\operatorname{Imm}_{(n,n)}\bigr)_{n\geq1}$.
Because the permanent family is VNP–complete, for any $f\in\mathrm{VNP}$,
\[
  f\leq_p(\operatorname{per}_n)_{n\geq1}\leq_p
  \bigl(\operatorname{Imm}_{(n,n)}\bigr)_{n\geq1}.
\]
Thus the family is VNP–hard. Since by Lemma \ref{lem2.3} it belongs to VNP, $\bigl(\operatorname{Imm}_{(n,n)}\bigr)_{n\geq1}$ is VNP–complete under $p$-projections. The identity concerning $U_n$ in the theorem has already been given by Lemma \ref{lem2.5}.\qed

\section{Conclusion}

This paper studies the computational complexity of immanants. We prove that the polynomial families $\bigl(\operatorname{Imm}_{(2^n)}\bigr)_{n\geq 1}$ and $\bigl(\operatorname{Imm}_{(n,n)}\bigr)_{n\geq 1}$ are both VNP-complete under $p$-projections. The computational complexity of immanants corresponding to other partitions deserves further investigation.

\noindent{\bf Data Availability}\\
{No data were used to support this study.}

\end{document}